\documentclass[reqno]{amsart}
\usepackage{amssymb, amsmath, amsthm, amscd, mathrsfs}

\usepackage{paralist}
\usepackage{cite}
\usepackage{bigints}
\usepackage{dsfont}
\usepackage{empheq}
\usepackage{amssymb}
\usepackage{cases}
\usepackage{enumitem}
\allowdisplaybreaks

\usepackage{subfig}

\usepackage{caption}
\usepackage{booktabs}

\usepackage{float}
\usepackage[ruled,vlined,linesnumbered,noresetcount]{algorithm2e}

\usepackage{hyperref}

\def\norm#1{\left\|#1\right\|}
\theoremstyle{plain}
\newtheorem{theorem}{Theorem}[section]

\newtheorem{lemma}[theorem]{Lemma}
\newtheorem{proposition}[theorem]{Proposition}

\theoremstyle{remark}
\newtheorem{remark}[theorem]{Remark}
\numberwithin{equation}{section}

\def \d {\mathrm{d}}

\title[Numerical identification of initial temperatures] 
      {Numerical identification of initial temperatures in heat equation with dynamic boundary conditions}

\author{S. E. Chorfi}
\author{G. El Guermai}
\author{L. Maniar}
\author{W. Zouhair}

\address{S. E. Chorfi, G. El Guermai, L. Maniar and W. Zouhair, Cadi Ayyad University, Faculty of Sciences Semlalia, LMDP, UMMISCO (IRD-UPMC), B.P. 2390, Marrakesh, Morocco}

\email{chorphi@gmail.com, ghita.el.guermai@gmail.com, maniar@uca.ma, walid.zouhair.fssm@gmail.com}

\makeatletter 
\@namedef{subjclassname@2020}{%
  \textup{2020} Mathematics Subject Classification}
\makeatother

\subjclass[2020]{Primary: 35R30, 49N45; Secondary: 35K05, 47A05.}
 \keywords{inverse problem, backward parabolic problem, dynamic boundary condition, adjoint problem, conjugate gradient.}

\begin{document}
\begin{abstract}
We investigate the inverse problem of numerically identifying unknown initial temperatures in a heat equation with dynamic boundary conditions whenever some overdetermination data is provided after a final time. This is a backward parabolic problem which is severely ill-posed. As a first step, the problem is reformulated as an optimization problem with an associated cost functional. Using the weak solution approach, an explicit formula for the Fr\'echet gradient of the cost functional is derived from the corresponding sensitivity and adjoint problems. Then the Lipschitz continuity of the gradient is proved. Next, further spectral properties of the input-output operator are established. Finally, the numerical results for noisy measured data are performed using the regularization framework and the conjugate gradient method. We consider both one- and two-dimensional numerical experiments using finite difference discretization to illustrate the efficiency of the designed algorithm. Aside from dealing with a time derivative on the boundary, the presence of a boundary diffusion makes the analysis more complicated. This issue is handled in the 2-D case by considering the polar coordinate system. The presented method implies fast numerical results.
\end{abstract}

\maketitle

\section{Introduction}
Inverse problems arise in various fields of science and engineering, including the modeling of physical phenomena such as heat transfer problems and diffusion processes, among others. Most of these problems are modeled by PDEs where some parameters are unknown and need to be identified from some measured data. This is exactly where the inverse problems theory comes in. More precisely, inverse problems theory consists of identifying some input parameters in a given system from a partial measurement on the solution, e.g., forcing terms \cite{Ha'07, Ha'11, Hun'21}, thermal conductivity and radiative coefficient \cite{Ha'21}, potential, damping coefficient and source terms \cite{KAMC, KAMC19, CGMZ'22}, or initial temperature \cite{LMZ'09}. Up to our knowledge, the most practical and realistic measurements considered in literature use the state at a fixed final time.

In this paper, we focus on identifying unknown initial temperatures in a heat equation with dynamic boundary conditions, from a noisy measurement of the final temperature. More precisely, let $T>0$ be a fixed final time and let $\Omega \subset \mathbb{R}^N$ be a bounded domain ($N\geq 1$ is an integer) with boundary $\Gamma=\partial\Omega$ of class $C^2$. We denote by $\Omega_T =(0,T)\times \Omega$ and $\Gamma_T =(0,T)\times \Gamma$. We consider the following heat equation with dynamic boundary conditions of surface diffusion type
\begin{empheq}[left = \empheqlbrace]{alignat=2}
\begin{aligned}
&\partial_t y -d \Delta y+a(x)y = 0, &&\qquad \text{in } \Omega_T , \\
&\partial_t y_{\Gamma} -\gamma \Delta_{\Gamma} y_{\Gamma}+d\partial_{\nu} y + b(x)y_{\Gamma} = 0, &&\qquad \text{on } \Gamma_T, \\
&y_{\Gamma}(t,x) = y_{|\Gamma}(t,x), &&\qquad \text{on } \Gamma_T, \\
&(y,y_{\Gamma})\rvert_{t=0}=(f, g),   &&\qquad \Omega\times\Gamma, \label{eq1to4}
\end{aligned}
\end{empheq}
where $\mathcal{G}:=(f, g)\in L^2(\Omega)\times L^2(\Gamma)$ is the unknown initial temperature. The diffusive constants are such that $d, \gamma > 0$ if $N\ge 2$, with the convention $\gamma=0$ if $N=1$. The radiative potentials are assumed to be bounded $a\in L^\infty(\Omega)$, $b\in L^\infty(\Gamma)$, and $D=\max(\|a\|_\infty,\|b\|_\infty)$. The Laplace operator is denoted by $\Delta=\Delta_x$. The trace of $y$ is $y_{\mid \Gamma}$, and the normal derivative is $\partial_{\nu} y:=(\nabla y \cdot \nu)_{|\Gamma}$, where $\nu(x)$ is the unit outward normal at $x\in \Gamma$. Let $\mathrm{g}$ be the standard Riemannian metric on $\Gamma$ induced by $\mathbb{R}^N$. We fix a coordinate system $x=(x^j)$ and we denote by $\left(\partial_{x^j}\right)$ the corresponding tangent vector field. The Laplace-Beltrami operator $\Delta_\Gamma$ is given locally by the formula
\begin{equation}\label{eqlb}
\Delta_\Gamma=\frac{1}{\sqrt{|\mathrm{g}|}} \sum_{i,j=1}^{N-1} \partial_{x^i} \left(\sqrt{|\mathrm{g}|}\, \mathrm{g}^{ij} \partial_{x^j}\right),
\end{equation}
where $\mathrm{g}=\left(\mathrm{g}_{ij}\right)$ still denotes the tensor matrix of $\mathrm{g}$, $\mathrm{g}^{-1}=\left(\mathrm{g}^{ij}\right)$ its inverse with determinant $|\mathrm{g}|=\det\left(\mathrm{g}_{ij}\right)$.

We refer to \cite{MMS'17,KM'19,FGGR'02} for a detailed exposition of dynamic boundary conditions. We also refer to the seminal paper \cite{Go'06}, where the physical derivation and interpretation of dynamic boundary conditions $\eqref{eq1to4}_2$ are discussed. Throughout the paper, we will often use the surface divergence formula
\begin{equation}\label{sdt}
\int_\Gamma \Delta_\Gamma y\, z \,\d S =- \int_\Gamma \langle \nabla_\Gamma y, \nabla_\Gamma z\rangle_\Gamma \,\d S, \quad y\in H^2(\Gamma), \, z\in H^1(\Gamma),
\end{equation}
where $\nabla_\Gamma y$ is the tangential gradient given by $\nabla_{\Gamma} y=\nabla y-\left(\partial_{\nu} y\right) \nu$, and $\langle \cdot, \cdot \rangle_\Gamma$ is the Riemannian inner product of tangential vectors on $\Gamma$.

The object of this paper is to numerically identify some unknown initial temperatures from the final ones. This is the so-called inverse initial data problem (or the backward parabolic problem) known as a typical problem which is severely ill-posed in the sense of Hadamard, that is, any arbitrarily small change in the initial datum lead to drastically huge change in the output solution \cite{VI}. Physically, this can be seen as a consequence of the time irreversibility and the smoothing effect of the heat process. 

Many contributions have been devoted to the enrichment of this field. For instance, the authors of \cite{WWMYBH} present a method based on reproducing kernels applied to a two-dimensional parabolic inverse source problem. The paper \cite{LYZCDYCH} investigate the simultaneous determination of an unknown initial temperature and a heat source from a given observation at a fixed internal location in heat equation with Dirichlet boundary conditions. The authors of \cite{KMMY} develop some numerical schemes based on positivity-preserving Padé approximations to solve an inverse initial data problem for heat equation with Dirichlet boundary conditions. In \cite{SSPRPNS}, the authors have been interested in recovering the initial temperature from the measurement on a part of the boundary. The paper \cite{KMFDZ} has considered a similar problem with a regularizing parameter that approximates and regularizes the heat conduction model. More recently, the authors of \cite{ABCEO} have analyzed numerically an inverse source problem in a degenerate parabolic equation. They have adapted an optimal control approach using finite element discretization and the augmented Lagrangian method. For more recent developments within this framework, we refer to the excellent book \cite{HR'21}.

The approach presented in this work can be seen as a continuation of the approach given in \cite{ACM1'21'} for the identification of source terms in heat equation with dynamic boundary conditions from the final time measured data. Compared to the aforementioned inverse source problem, the ill-posedness of the inverse initial temperatures problem is more severe. In view of the stability results in \cite{ACM'21, ACMO'20}, one can expect a weak convergence rate and slow numerical results compared to the fast numerical results for the inverse source problem obtained in \cite{ACM1'21'} by the Landweber iteration. In fact, in \cite{ACMO'20}, the authors have obtained a Lipschitz stability estimate for the source terms, while in \cite{ACM'21} only a logarithmic stability estimate is proved for the initial temperatures. It should be pointed out that such conditional stability results serve to suitably select the regularizing parameter and determine the convergence rate of the regularized solution to an exact solution \cite{CY'00}. As a remedy to this issue, we propose a modified algorithm based on the conjugate gradient (CG) method to speed up the convergence process. We refer to \cite{EHN'00} for a detailed exposition of such algorithms.

To the best of our knowledge, the problem considered in this paper has only been shown to be solvable in a particular form of static boundary conditions (Dirichlet, Neumann or Robin); see, for instance, \cite{KMFDZ,SSPRPNS,KMMY}. However, for more general boundary conditions as dynamic boundary conditions, the problem has not been studied, and the numerical reconstruction needs to be investigated. Therefore, the main objective of the present work is to extend the weak solution strategy for the determination of initial temperatures to undertake the numerical aspect of such a problem.

The paper is structured as follows: in Section \ref{sec2}, we recall some well-posedness and regularity results related to system \eqref{eq1to4} that are needed in the sequel. In Section \ref{sec3}, we study the minimization problem associated to the cost functional. Next, we prove an explicit formula for the gradient of the cost functional. This is done by introducing suitable adjoint and sensitivity systems. Then the Lipschitz continuity of the gradient is obtained. In Section \ref{sec5}, the gradient formula is algorithmically implemented via the conjugate gradient method for the numerical reconstruction of initial temperatures in one- and two-dimensional systems. Finally, Section \ref{sec6} is devoted to some conclusions.

\section{Wellposedness of the Cauchy problem}\label{sec2}
We start by recalling some relevant results concerning the wellposedness and the regularity of the solution of system \eqref{eq1to4}. For convenience, we refer to \cite{MMS'17} for more details.

We consider the Lebesgue measure $\d x$ on $\Omega$ and the surface measure $\d S$ on $\Gamma$. Let us introduce the functional spaces
$$\mathbb{L}^2:=L^2(\Omega, \d x)\times L^2(\Gamma, \d S)\qquad \text{and}\qquad \mathbb{L}_T^2:=L^2(\Omega_T)\times L^2(\Gamma_T).$$
Note that $\mathbb{L}^2$ and $\mathbb{L}_T^2$ are Hilbert spaces respectively endowed with the scalar products defined by
\begin{align*}
\langle (y,y_\Gamma),(z,z_\Gamma)\rangle &=\langle y,z\rangle_{L^2(\Omega)} +\langle y_\Gamma,z_\Gamma\rangle_{L^2(\Gamma)},
\end{align*}
\begin{align*}
\langle (y,y_\Gamma),(z,z_\Gamma)\rangle_{\mathbb{L}_T^2} &=\langle y,z\rangle_{L^2(\Omega_T)} +\langle y_\Gamma,z_\Gamma\rangle_{L^2(\Gamma_T)}.
\end{align*}
We further introduce the functional space $$\mathbb{H}^k:=\left\{(y,y_\Gamma)\in H^k(\Omega)\times H^k(\Gamma)\colon y_{|\Gamma} =y_\Gamma \right\} \text{ for } k=1,2,$$ endowed with the standard product norm.

The system \eqref{eq1to4} can be written as an abstract Cauchy  problem as follows
\begin{equation}\label{acp}
\text{(ACP)	} \; \begin{cases}
\hspace{-0.1cm} \partial_t Y=\mathcal{A} Y, \quad 0<t \le T, \nonumber\\
\hspace{-0.1cm} Y(0)=\mathcal{G}:=(f, g), \nonumber
\end{cases}
\end{equation}
where $Y:=(y,y_{\Gamma})$, and $\mathcal{A} \colon D(\mathcal{A}) \subset \mathbb{L}^2 \longrightarrow \mathbb{L}^2$ is a linear operator defined by
\begin{equation*}
\mathcal{A}=\begin{pmatrix} d\Delta - a & 0\\ -d\partial_\nu & \gamma \Delta_\Gamma - b\end{pmatrix}, \qquad \qquad D(\mathcal{A})=\mathbb{H}^2.
\end{equation*}
The operator $\mathcal{A}$ is a bounded perturbation of the linear operator $\mathcal{A}_0 \colon D(\mathcal{A}_0) \subset \mathbb{L}^2 \longrightarrow \mathbb{L}^2$ defined by
\begin{equation*}
\mathcal{A}_0=\begin{pmatrix} d\Delta & 0\\ -d\partial_\nu & \gamma \Delta_\Gamma\end{pmatrix}, \qquad \qquad D(\mathcal{A}_0)=\mathbb{H}^2.
\end{equation*}
It has been shown in \cite{MMS'17} that the operator $\mathcal{A}_0$ is the infinitesimal generator of an analytic $C_0$-semigroup on $\mathbb{L}^2$. Thus, the system \eqref{eq1to4} is well-posed and enjoys the maximal regularity property by a standard perturbation argument.

The following spectral lemma will be useful in the sequel.
\begin{lemma}\label{lmsp1}
There exists an increasing sequence of eigenvalues $(\lambda_n)_{n \ge 1}$ that constitutes the spectrum of $-\mathcal{A}_0$ such that $0= \lambda_1 < \lambda_2 < \cdots$, and $\lambda_n \to \infty$ as $n\to \infty$. Moreover, the sequence of the corresponding eigenfunctions $(\phi_n)_{n\ge 1}$ forms an orthonormal basis of $\mathbb{L}^2$.
\end{lemma}

\begin{proof}
By \cite[Proposition 2.1]{MMS'17}, the operator $-\mathcal{A}_0$ is self-adjoint and nonnegative. Moreover, since $\mathbb{H}^{2} \hookrightarrow H^{2}(\Omega) \times H^{2}(\Gamma)$ is continuous, then $\mathbb{H}^{2} \hookrightarrow \mathbb{L}^2$ is compact. By \cite[Proposition 4.25]{EN'99}, we deduce that $-\mathcal{A}_0$ has a compact resolvent. Now, the existence of $(\lambda_n)_{n\ge 1}$ and $(\phi_n)_{n\ge 1}$ follows from the standard spectral theory.
\end{proof}

\section{Gradient formula and Lipschitz continuity of the gradient}\label{sec3}
In this section, we consider the following inverse initial data problem.

\noindent\textbf{Backward Parabolic Problem (BPP).} The initial temperature $\mathcal{G}:=(f,g)$ in the system \eqref{eq1to4} is unknown and to be restored from the measured final temperature
$$Y_{T}:=\left(y(T,\cdot),y_\Gamma(T,\cdot)\right)\in \mathbb{L}^2.$$
The measured data is assumed to contain a random noise.

Let $Y(t,\cdot,\mathcal{G})$ be the mild solution of \eqref{eq1to4} corresponding to the initial datum $\mathcal{G}=(f,g)$. We introduce the input-output operator $\Psi \colon \mathbb{L}^2 \longrightarrow \mathbb{L}^2$ as follows
$$(\Psi \mathcal{G})(\cdot)=Y_T(\cdot).$$
Then the BPP can be reformulated as the following operator equation
\begin{equation}\label{2eq2.3}
\Psi \mathcal{G}=Y_T, \quad Y_T \in \mathbb{L}^2.
\end{equation}
The following lemma shows that the operator $\Psi \colon \mathbb{L}^2 \longrightarrow \mathbb{L}^2$ is compact. 
\begin{lemma}\label{2lem2.4}
The input-output operator 
\begin{equation*}
\Psi \colon \mathbb{L}^2 \ni \mathcal{G} \longmapsto Y(T,\cdot,\mathcal{G}) \in \mathbb{L}^2
\end{equation*}
is a linear compact operator.
\end{lemma}
\begin{proof}
The proof is similar to the proof of \cite[Lemma 2]{ACM1'21'}.
\end{proof}

Consequently, the inverse problem \eqref{2eq2.3} is ill-posed (unstable). Therefore, we define a quasi-solution $\mathcal{G }_*$ of the BPP as a minimizer of the problem
\begin{align}
&\mathcal{J}(\mathcal{G}_*)= \inf_{\mathcal{G}\in \mathbb{L}^2} \mathcal{J}(\mathcal{G}), \label{neq2.5}\\
\mathcal{J}(\mathcal{G})&=\frac{1}{2} \|Y(T, \cdot,\mathcal{G})-Y_{T}^\delta\|_{\mathbb{L}^2}^2, \qquad \mathcal{G}\in \mathbb{L}^2, \label{neq2.6}
\end{align}
where $Y_{T}^\delta=(y_{T}^\delta, y_{T, \Gamma}^\delta)$ is a noisy data of $Y_{T}$ such that $\|Y_{T}-Y_{T}^\delta\|\leq \delta$ for $\delta \geq 0$.

As a remedy to the ill-posedness of the BPP, the regularized Tikhonov functional is defined as
\begin{align}\label{regpb}
\mathcal{J}_\varepsilon(\mathcal{G})&=\frac{1}{2} \|Y(T, \cdot,\mathcal{G})-Y_{T}^\delta\|_{\mathbb{L}^2}^2 +\frac{\varepsilon}{2} \|\mathcal{G}\|_{\mathbb{L}^2}^2, \qquad \mathcal{G}\in \mathbb{L}^2, 
\end{align}
where $\varepsilon >0$ is the regularizing parameter.

Although the problem is unstable, one can obtain some conditional stability provided that a priori bound on unknown initial data is known. This is shown in the following lemma using the logarithmic convexity method \cite{Pa'75}.
\begin{lemma}
The solution $Y(t)$ of the system \eqref{eq1to4} satisfies the following estimate
\begin{equation}\label{lc}
    \|Y(t)\|_{\mathbb{L}^2} \le \|Y(0)\|_{\mathbb{L}^2}^{1-\frac{t}{T}} \|Y(T)\|_{\mathbb{L}^2}^{\frac{t}{T}}, \qquad 0\le t \le T.
\end{equation}
\end{lemma}

\begin{proof}
Let $\mathcal{G}:=(f,g)\ne (0,0)$. Let us define the energy of the system \eqref{eq1to4} as
$$E(t)=\frac{1}{2} \|Y(t)\|_{\mathbb{L}^2}^2, \qquad t \in [0,T].$$
Using the equations of system \eqref{eq1to4} and integrating by parts in $\Omega$ and on $\Gamma$, we obtain
\begin{align*}
E^{\prime}(t) &= \left\langle Y(t), \partial_t Y(t)\right\rangle_{\mathbb{L}^2}\\
&= \int_\Omega y \partial_t y \,\d x + \int_\Gamma y_\Gamma \partial_t y_\Gamma\, \d S\\
&= \int_\Omega y (d\Delta y - a y) \,\d x + \int_\Gamma y_\Gamma (\gamma \Delta_\Gamma y_\Gamma -d \partial_\nu y - by_\Gamma) \, \d S\\
& = -d \int_\Omega |\nabla y|^2 \,\d x - \int_\Omega a y^2 \,\d x -\gamma \int_\Gamma |\nabla_\Gamma y_\Gamma|^2 \, \d S - \int_\Gamma b y_\Gamma^2 \,\d S.
\end{align*}
Differentiating in time $t$ and integrating by parts again using \eqref{eq1to4}, we obtain
\begin{align*}
E^{\prime\prime}(t)&= 2 \int_\Omega \left(- d\nabla y\cdot \nabla (\partial_t y) -a y \partial_t y\right)\,\d x + 2  \int_\Gamma \left(-\gamma\nabla_\Gamma y_\Gamma\cdot \nabla_\Gamma (\partial_t y_\Gamma) -b y_\Gamma \partial_t y_\Gamma\right)\, \d S \\
&=2 \|\partial_t Y(t)\|_{\mathbb{L}^2}^2.
\end{align*}
By Cauchy-Schwarz inequality, we infer that
$$E(t) E^{\prime \prime}(t) -(E^{\prime}(t))^{2}=\|Y(t)\|_{\mathbb{L}^2}^2 \|\partial_t Y(t)\|_{\mathbb{L}^2}^2 -\left\langle Y(t), \partial_t Y(t)\right\rangle^2_{\mathbb{L}^2} \ge 0.$$
It follows then that
$$(\log E(t))^{\prime \prime}=\frac{E(t) E^{\prime \prime}(t) -(E^{\prime}(t))^{2}}{E^2(t)} \geq 0.$$
This implies that $\log E(t)$ is a convex function of $t$, which yields
$$\log E(t) \leq\left(1-\frac{t}{T}\right) \log E(0)+\frac{t}{T} \log E(T).$$
This completes the proof.
\end{proof}
Now if we assume that $\|Y(0)\|_{\mathbb{L}^2} \le M$ in \eqref{lc}, where $M$ is a given positive constant, the following conditional stability holds
$$\|Y(t)\|_{\mathbb{L}^2} \le M^{1-\frac{t}{T}} \|Y(T)\|_{\mathbb{L}^2}^{\frac{t}{T}}, \qquad 0\le t \le T.$$
Such a result would be of practical interest if one can compute the prior bound $M$ from the accessible data of the physical model.

Next, we derive a gradient formula for $\mathcal{J}$ via the mild solution $\Phi=(\varphi, \varphi_\Gamma)$ of an adjoint system.
\begin{proposition}\label{2prop2.6}
The cost functional $\mathcal{J}$ is Fr\'echet differentiable and its gradient at each $\mathcal{G}\in \mathbb{L}^2$ is given by
\begin{equation}\label{2eq2.11}
\mathcal{J}'(\mathcal{G})(x)= \Phi(x,0,\mathcal{G}) \qquad \text{for a.e }\, x\in \Omega \times \Gamma,
\end{equation}
where $\Phi(t,\cdot,\mathcal{G})=(\varphi, \varphi_\Gamma)$ is the mild solution of the following adjoint system
\begin{empheq}[left = \empheqlbrace]{alignat=2}
\begin{aligned}
&-\partial_t \varphi -d \Delta \varphi +a(x)\varphi = 0, &&\qquad \text{in } \Omega_T , \\
&-\partial_t \varphi_{\Gamma} -\gamma \Delta_{\Gamma} \varphi_{\Gamma} +d\partial_{\nu} \varphi +b(x)\varphi_{\Gamma} = 0, &&\qquad\text{on } \Gamma_T, \\
&\varphi_{\Gamma}(t,x) = \varphi_{|\Gamma}(t,x), && \qquad\text{on } \Gamma_T, \\
&\varphi\rvert_{t=T} =y(T, \cdot, \mathcal{G})-y_{T}^\delta,   &&\qquad \text{in } \Omega ,\\
&\varphi_{\Gamma}\rvert_{t=T} = y_\Gamma(T, \cdot, \mathcal{G})-y_{T, \Gamma}^\delta, && \qquad\text{on } \Gamma .
\label{aeq1to5}
\end{aligned}
\end{empheq}
\end{proposition}
\begin{proof}
Let $\mathcal{G}, \mathcal{G}+ \delta \mathcal{G}\in \mathbb{L}^2$ be given initial data and $Y(x,t,\mathcal{G})$, $Y(x,t,\mathcal{G}+\delta\mathcal{G})$ be corresponding solutions of \eqref{eq1to4}. Hence, By linearity of the systems, $$\delta Y(x,t,\mathcal{G}):=Y(x,t,\mathcal{G}+\delta\mathcal{G})-Y(x,t,\mathcal{G}),$$ is the mild solution of the following system
\begin{empheq}[left = \empheqlbrace]{alignat=2}
\begin{aligned}
&\partial_t \delta y -d \Delta \delta y+a(x)\delta y = 0, &&\qquad \text{in } \Omega_T , \\
&\partial_t \delta y_{\Gamma} -\gamma \Delta_{\Gamma} \delta y_{\Gamma}+d\partial_{\nu} \delta y + b(x)\delta y_{\Gamma} = 0, &&\qquad \text{on } \Gamma_T, \\
&\delta y_{\Gamma}(t,x) =\delta y_{|\Gamma}(t,x), &&\qquad \text{on } \Gamma_T, \\
&(\delta y,\delta y_{\Gamma})\rvert_{t=0}=(\delta f,\delta g),   &&\qquad \Omega\times\Gamma \label{eq12}
\end{aligned}
\end{empheq}
Let us compute the difference
$$\delta \mathcal{J}(\mathcal{G}):=\mathcal{J}(\mathcal{G}+\delta \mathcal{G}) -\mathcal{J}(\mathcal{G}).$$
We have
\begin{align*}
\delta \mathcal{J}(\mathcal{G})&= \frac{1}{2} \|Y(T, \cdot, \mathcal{G}+ \delta \mathcal{G})-Y_{T}^\delta\|_{\mathbb{L}^2}^2 -\frac{1}{2} \|Y(T, \cdot,\mathcal{G})-Y_{T}^\delta\|_{\mathbb{L}^2}^2, \nonumber\\
&= \frac{1}{2} \left(\|y(T, \cdot, \mathcal{G}+ \delta \mathcal{G})-y_{T}^\delta\|_{L^2(\Omega)}^2 + \|y_\Gamma(T, \cdot, \mathcal{G}+ \delta \mathcal{G})-y_{T, \Gamma}^\delta\|_{L^2(\Gamma)}^2\right) \nonumber\\
& \qquad - \frac{1}{2} \left(\|y(T, \cdot, \mathcal{G})-y_{T}^\delta\|_{L^2(\Omega)}^2 + \|y_\Gamma(T, \cdot, \mathcal{G})-y_{T, \Gamma}^\delta\|_{L^2(\Gamma)}^2\right) \nonumber\\
&= \frac{1}{2} \int_\Omega \left[(y(T, x, \mathcal{G}+ \delta \mathcal{G})-y_{T}^\delta(x))^2  - (y(T, x, \mathcal{G})-y_{T}^\delta(x))^2 \right] \,\d x \\
& \quad + \frac{1}{2} \int_\Gamma \left[(y_\Gamma(T, x, \mathcal{G}+ \delta \mathcal{G})-y_{T, \Gamma}^\delta(x))^2 - (y_\Gamma(T, x, \mathcal{G})-y_{T, \Gamma}^\delta)^2 \right] \,\d S.
\end{align*}
Using the identity
$$\frac{1}{2} \left[(x-z)^2-(y-z)^2\right]=(y-z)(x-y)+\frac{1}{2}(x-y)^2, \; x,y\in \mathbb{R},$$
in the last two terms, we obtain
\begin{align}
\delta \mathcal{J}(\mathcal{G}) &= \int_\Omega (y(T, x, \mathcal{G})-y_{T}^\delta(x)) \delta y(T, x, \mathcal{G}) \,\d x +\frac{1}{2} \int_\Omega [\delta y(T, x, \mathcal{G})]^2 \,\d x \label{2eq2.19}\\
& \hspace{-0.3cm} + \int_\Gamma (y_\Gamma(T, x, \mathcal{G})-y_{T, \Gamma}^\delta(x)) \delta y_\Gamma(T, x, \mathcal{G}) \,\d S +\frac{1}{2} \int_\Gamma [\delta y_\Gamma(T, x, \mathcal{G})]^2 \,\d S \label{2eq2.20},
\end{align}
We rewrite the first integral in the right-hand side of \eqref{2eq2.19} using $\Phi(t,\cdot,\mathcal{G})$ and $\delta Y(t,\cdot, \mathcal{G})$, the mild solutions of \eqref{aeq1to5} and \eqref{eq12} respectively. We have
\begin{align}
&\int_\Omega (y(T, x, \mathcal{G})-y_{T}^\delta(x)) \delta y(T, x, \mathcal{G}) \,\d x \nonumber =\int_\Omega \varphi(T, x, \mathcal{G}) \delta y(T, x, \mathcal{G}) \,\d x \nonumber\\
&= \bigintss_\Omega \left[\int_0^T \partial_t (\varphi(t, x, \mathcal{G}) \delta y(t, x, \mathcal{G})) \,\d t \right]\,\d x+\int_{\Omega}\varphi(0,x)\,\delta f(x)\d x \nonumber\\
&= \int_{\Omega_T} \left[(\partial_t \varphi) \delta y + \varphi \partial_t(\delta y) \right] \,\d x \,\d t+\int_{\Omega}\varphi(0,x)\,\delta f(x)\d x \nonumber\\
&= \int_{\Omega_T} \left[ (-d \Delta \varphi +a(x)\varphi) \delta y + \varphi (d \Delta (\delta y)-a(x)(\delta y)) \right]\,\d x \,\d t \nonumber\\
& \qquad +\int_{\Omega}\varphi(0,x)\,\delta f(x)\d x \nonumber\\
&= \int_{\Omega_T} -d[(\Delta \varphi) \delta y -\Delta(\delta y)\varphi ]\,\d x \, \d t +\int_{\Omega}\varphi(0,x)\,\delta f(x)\d x \nonumber\\
&= \int_{\Gamma_T} -d[(\partial_\nu \varphi) \delta y_\Gamma -\varphi_\Gamma \partial_\nu  (\delta y)] \,\d S \, \d t +\int_{\Omega}\varphi(0,x)\,\delta f(x)\d x \label{2eq2.25}.
\end{align}
Similarly, for the first integral in the right-hand side of \eqref{2eq2.20}, we obtain
\begin{align}
&\int_\Gamma (y_\Gamma(T, x, \mathcal{G})-y_{T, \Gamma}^\delta(x)) \delta y_\Gamma(T, x, \mathcal{G}) \,\d S \nonumber\\
&= \int_{\Gamma_T} -\gamma[(\Delta_\Gamma \varphi_\Gamma) \delta y_\Gamma -\Delta_\Gamma(\delta y_\Gamma)\varphi_\Gamma] \,\d S \,\d t \label{2eq2.26}\\
&\qquad + \int_{\Gamma_T} d[(\partial_\nu \varphi) \delta y_\Gamma -\varphi_\Gamma \partial_\nu  (\delta y)] \,\d S \,\d t + \int_{\Gamma} \varphi_\Gamma(0,x)\, \delta g(x) \,\d S \nonumber\\
&= \int_{\Gamma_T} d[(\partial_\nu \varphi) \delta y_\Gamma -\varphi_\Gamma \partial_\nu  (\delta y)] \,\d S \,\d t+ \int_{\Gamma} \varphi_\Gamma(0,x) \,\delta g(x) \,\d S  \label{2eq2.27}
\end{align}
The first integral in the right-hand side of \eqref{2eq2.26} is null by the surface divergence formula \eqref{sdt}. Adding up the two integrals \eqref{2eq2.25} and \eqref{2eq2.27}, we obtain
\begin{align}
&\int_\Omega (y(T, x, \mathcal{G})-y_{T}^\delta(x)) \delta y(T, x, \mathcal{G}) \,\d x + \int_\Gamma (y_\Gamma(T, x, \mathcal{G})-y_{T, \Gamma}^\delta(x)) \delta y_\Gamma(T, x, \mathcal{G}) \,\d S \nonumber\\
&= \int_{\Omega_T}\varphi(0,x) \, \delta f(x)\,\d x + \int_{\Gamma_T}\varphi_\Gamma(0,x) \, \delta g(x)\,\d S . \nonumber
\end{align}
The second integrals in the right-hand sides of \eqref{2eq2.19} and \eqref{2eq2.20} are of order $\mathcal{O}\left(\|\delta \mathcal{G}\|^2_{\mathbb{L}^2}\right)$. In fact, multiplying \eqref{eq12}$_1$ by $\delta y$, \eqref{eq12}$_2$ by $\delta y_\Gamma$ and using Green and surface divergence formula, we obtain
\begin{equation}
    \dfrac{1}{2}\partial_t\left(\int_{\Omega} |\delta y|^2\, \d x  \right)+d\int _{\Omega}|\nabla(\delta y)|^2\,\d x+\dfrac{d}{2}\int_{\Gamma}\partial_{\nu}(|\delta y|^2)\, \d S+\int_{\Omega} a |\delta y|^2 \, \d x =0,
    \label{equa1}
\end{equation}
\begin{equation}
    \dfrac{1}{2}\partial_t\left(\int_{\Gamma} |\delta y_{\Gamma}|^2\, \d S  \right)+\gamma\int _{\Gamma}|\nabla_{\Gamma}(\delta y_{\Gamma})|^2\,\d S-\dfrac{d}{2}\int_{\Gamma}\partial_{\nu}(|\delta y|^2)\, \d S+\int_{\Gamma} b |\delta y_{\Gamma}|^2 \, \d S=0.
    \label{equa2}
\end{equation}
Combining \eqref{equa1} and \eqref{equa2}, we arrive at
\begin{align*}
    &\dfrac{1}{2}\dfrac{\d}{\d t}\|\delta Y(t)\|_{\mathbb{L}^2}^2+d\int _{\Omega}|\nabla(\delta y)|^2\,\d x+\gamma\int _{\Gamma}|\nabla_{\Gamma}(\delta y_{\Gamma})|^2\,\d S+\int_{\Omega} a |\delta y|^2 \, \d x \\
    &\qquad+\int_{\Gamma} b |\delta y_{\Gamma}|^2 \, \d S=0.
\end{align*}
By Cauchy-Schwarz inequality, we obtain 
\begin{equation*}
    \dfrac{1}{2}\dfrac{\d}{\d t}\|\delta Y(t)\|_{\mathbb{L}^2}^2 +\min (d,\gamma)\|(\nabla(\delta y),\nabla_{\Gamma}(\delta y_{\Gamma}))\|_{\mathbb{L}^2}^2\leq D\|\delta Y(t)\|_{\mathbb{L}^2}^2.
\end{equation*}
Therefore, 
\begin{equation*}
    \dfrac{1}{2}\dfrac{\d}{\d t}\|\delta Y(t)\|_{\mathbb{L}^2}^2 \leq D\|\delta Y(t)\|_{\mathbb{L}^2}^2.
\end{equation*}
Then Gronwall inequality imply that 
\begin{equation}
    \|\delta Y(t)\|_{\mathbb{L}^2}^2 \leq \mathrm{e}^{2DT}\|\delta Y(0)\|_{\mathbb{L}^2}^2=\mathrm{e}^{2DT}\|\delta \mathcal{G}\|_{\mathbb{L}^2}^2,
    \label{inequa1}
\end{equation}
for every $t\in [0,T]$. Consequently,
$$\frac{1}{2} \int_\Omega [\delta y(T, x, \mathcal{G})]^2 \,\d x +\frac{1}{2} \int_\Gamma [\delta y_\Gamma(T, x, \mathcal{G})]^2 \,\d S =\mathcal{O}\left(\|\delta \mathcal{G}\|^2_{\mathbb{L}^2}\right).$$
This completes the proof.
\end{proof}

Now, we show the Lipschitz continuity of the Fréchet gradient $\mathcal{J}'$.
\begin{lemma}\label{2lem2.8}
Let $\mathcal{G}, \delta \mathcal{G}\in \mathbb{L}^2$. The Fréchet gradient $\mathcal{J}'$ is Lipschitz continuous,
\begin{equation*}
\|\mathcal{J}'(\mathcal{G}+ \delta \mathcal{G})-\mathcal{J}'(\mathcal{G})\|_{\mathbb{L}^2} \leq L \|\delta \mathcal{G}\|_{\mathbb{L}^2}, 
\end{equation*}
where the Lipschitz constant $L>0$ depends on $T$ and $D$ as follows
\begin{equation}\label{lip}
L=\sqrt{\mathrm{e}^{2DT}\left(1+2TD\mathrm{e}^{2TD }\right)}. 
\end{equation}
\end{lemma}

\begin{proof}
By Proposition \eqref{2prop2.6}, we have 
\begin{equation}
    \|\mathcal{J}'(\mathcal{G}+ \delta \mathcal{G})-\mathcal{J}'(\mathcal{G})\|^2_{\mathbb{L}^2}:=\|\delta \Phi(0,\cdot,\mathcal{G})\|_{\mathbb{L}^2}^2,
    \label{equa5}
\end{equation}
where $\delta \Phi(t, \cdot, \mathcal{G}):=\Phi(t, \cdot, \mathcal{G}+\delta\mathcal{G})-\Phi(t, \cdot, \mathcal{G})$ be the strong solution of the adjoint system
\begin{empheq}[left = \empheqlbrace]{alignat=2}
\begin{aligned}
&-\partial_t (\delta \varphi) -d \Delta (\delta \varphi)+ a(x)(\delta \varphi) = 0, &\qquad \text{in } \Omega_T , \\
&-\partial_t (\delta \varphi_{\Gamma}) -\gamma \Delta_{\Gamma} (\delta \varphi_{\Gamma}) + d\partial_{\nu} (\delta \varphi) + b(x)(\delta \varphi_{\Gamma}) = 0, &\qquad \text{on } \Gamma_T, \\
&\delta \varphi_{\Gamma}(t,x) = (\delta \varphi)_{|\Gamma}(t,x), &\qquad \text{on } \Gamma_T, \\
&(\delta \varphi, \delta \varphi_{\Gamma})\rvert_{t=T}=(\delta y, \delta y_{\Gamma})\rvert_{t=T},   &\qquad \Omega \times\Gamma.
\label{2aeq1to4}
\end{aligned}
\end{empheq}
Next, we estimate the norm $\|\delta \Phi(0,\cdot,\mathcal{G})\|_{\mathbb{L}^2}^2$. In the adjoint system \eqref{2aeq1to4}, multiplying the first equation by $\delta \varphi$ and the second by $\delta \varphi_\Gamma$, we obtain the following identities
\begin{align*}
& -\frac{1}{2} \frac{\mathrm{d}}{\mathrm{d}t} \left( \int_\Omega |\delta \varphi|^2 \,\d x\right) +d\int_\Omega |\nabla(\delta \varphi)|^2 \,\d x -\frac{1}{2} \int_\Gamma d\partial_\nu (|\delta \varphi|^2) \,\d S + \int_\Omega a |\delta \varphi|^2 \,\d x =0, \\
& -\frac{1}{2} \frac{\mathrm{d}}{\mathrm{d}t} \left( \int_\Gamma |\delta \varphi_\Gamma|^2 \,\d S\right) +\gamma  \int_\Gamma |\nabla_\Gamma(\delta \varphi_\Gamma)|^2 \,\d S +\frac{1}{2} \int_\Gamma d\partial_\nu (|\delta \varphi|^2) \,\d S + \int_\Gamma b |\delta \varphi_\Gamma|^2 \,\d S =0, 
\end{align*}
where we employed Green and the surface divergence formulae. Combining the last identities, we obtain
\begin{align}
&\frac{1}{2} \frac{\mathrm{d}}{\mathrm{d}t} \left( \int_\Omega |\delta \varphi|^2 \,\d x + \int_\Gamma |\delta \varphi_\Gamma|^2 \,\d S\right)-\int_\Omega a |\delta \varphi|^2 \,\d x -\int_\Gamma b |\delta \varphi_\Gamma|^2 \,\d S  \nonumber\\
&= d\int_\Omega |\nabla(\delta \varphi)|^2 \,\d x + \gamma  \int_\Gamma |\nabla_\Gamma(\delta \varphi_\Gamma)|^2 \,\d S .
\label{equa3}
\end{align}
Then
\begin{align*}
& \dfrac{\mathrm{d}}{\mathrm{d}t} \|\delta \Phi (t,\cdot,\mathcal{G})\|_{\mathbb{L}^2}^2-2\int_\Omega a |\delta \varphi|^2 \,\d x -2\int_\Gamma b |\delta \varphi_\Gamma|^2 \,\d S \geq 0.
\end{align*}
Integrating between $0$ and $T$ and using the fact that $\delta \Phi(T,\cdot,\mathcal{G})=\delta Y(T,\cdot,\mathcal{G})$, we obtain
\begin{equation}
\|\delta \Phi (0,\cdot,\mathcal{G})\|_{\mathbb{L}^2}^2\leq \|\delta Y(T,\cdot,\mathcal{G})\|_{\mathbb{L}^2}^2+2D\|\Phi\|_{\mathbb{L}_{T}^2}^2.
\label{inequa3}
\end{equation}
Now, we estimate $\|\Phi\|_{\mathbb{L}_{T}^2}^2$, by \eqref{equa3}, we obtain
\begin{equation*}
    \dfrac{\mathrm{d}}{\mathrm{d}t}\|\delta \Phi (t,\cdot,\mathcal{G})\|_{\mathbb{L}^2}^2+2D\|\delta \Phi (t,\cdot,\mathcal{G})\|_{\mathbb{L}^2}^2\geq 0.
\end{equation*}
This inequality implies that the function $h$ define by 
\begin{equation*}
    h(t)=\mathrm{e}^{2D t}\|\delta \Phi (t,\cdot,\mathcal{G})\|_{\mathbb{L}^2}^2,
\end{equation*}
is non-decreasing on $[0,T]$. Therefore,
\begin{align*}
    \|\delta \Phi (t,\cdot,\mathcal{G})\|_{\mathbb{L}_{T}^2}^2&= \int_{0}^T \|\delta \Phi (t,\cdot,\mathcal{G})\|_{\mathbb{L}^2}^2 \, \d t\\
    & \leq \int_{0}^T \mathrm{e}^{2Dt } \|\delta \Phi (t,\cdot,\mathcal{G})\|_{\mathbb{L}^2}^2 \, \d t\\
    & \leq T\mathrm{e}^{2DT }\|\delta \Phi (T,\cdot,\mathcal{G})\|_{\mathbb{L}^2}^2.
\end{align*}
Using the last inequality and \eqref{2aeq1to4}, we arrive at
\begin{equation}
    \|\delta \Phi (t,\cdot,\mathcal{G})\|_{\mathbb{L}_{T}^2}^2 \leq T\mathrm{e}^{2DT }\|\delta Y (T,\cdot,\mathcal{G})\|_{\mathbb{L}^2}^2.
    \label{inequa4}
\end{equation}
From \eqref{inequa3} and \eqref{inequa4}, we obtain 
\begin{equation*}
\|\delta \Phi (0,\cdot,\mathcal{G})\|_{\mathbb{L}^2}^2\leq  \left(1+2TD\mathrm{e}^{2DT }\right)\|Y(T,\cdot,\mathcal{G})\|_{\mathbb{L}^2}^2.
\end{equation*}
Using inequality \eqref{inequa1}, we obtain
\begin{align*}
&\|\delta \Phi (0,\cdot,\mathcal{G})\|_{\mathbb{L}^2}^2\leq \mathrm{e}^{2DT} \left(1+2TD\mathrm{e}^{2DT }\right)\|\delta \mathcal{G}\|_{\mathbb{L}^2}^2.
\end{align*}
With the equality \eqref{equa5}, this implies that
\begin{align*}
&\|\mathcal{J}'(\mathcal{G}+ \delta \mathcal{G})-\mathcal{J}'(\mathcal{G})\|^2_{\mathbb{L}^2_T} \leq 2 \mathrm{e}^{2DT} \left(1+2TD\mathrm{e}^{2DT }\right)\|\delta \mathcal{G}\|_{\mathbb{L}^2}^2.
\end{align*}
This yields the desired result. 
\end{proof}

\section{Spectral analysis of the BPP}
Here, we will give a necessary and sufficient condition for the existence of the solution to BPP. We also provide a series representation of the unique solution to \eqref{regpb} for the system \eqref{eq1to4} when $a=0$ and $b=0$. For this purpose, we need the following lemma.
\begin{lemma}
The input-output operator $\Psi \colon \mathbb{L}^2 \longrightarrow \mathbb{L}^2$ defined in \eqref{2eq2.3} is self-adjoint. Moreover,
\begin{equation}\label{eq5.3}
   \left( \Psi\phi_k \right)(x) = \mathrm{e}^{- \lambda_{k} T} \phi_k (x) \qquad \text{ in } \Omega\times \Gamma,
\end{equation}
where $(\phi_k)_{k\ge 1}$ is the sequence pf eigenfunctions of the operator $-\mathcal{A}_0$ corresponding to the eigenvalues $(\lambda_k)_{k\geq1}$ given by Lemma \ref{lmsp1}.
\end{lemma}

\begin{proof}
By virtue of the orthonormal basis $(\phi_k)_k \subset \mathbb{L}^2$, the solution of the system $\text{(ACP)}$ corresponding to the initial data $\mathcal{G}$ and the operator $\mathcal{A}_0$  can be written as
\begin{equation*}
  Y(t,x) = \sum_{k = 1}^{\infty} \left\langle \mathcal{G},\phi_k   \right\rangle \mathrm{e}^{- \lambda_{k} t} \phi_k \qquad \text{ in } \Omega\times \Gamma.
\end{equation*}
Then, the Fourier series representation of the input-output operator is given by
\begin{equation}\label{eq5.4}
(\Psi \mathcal{G})(x)=Y(T,x,\mathcal{G}) = \sum_{k = 1}^{\infty} \left\langle \mathcal{G},\Phi_k   \right\rangle  \mathrm{e}^{- \lambda_{k} T} \phi_k (x), 
\end{equation}
Thus,
\begin{equation*}
    (\Psi \phi_n)(x) = \mathrm{e}^{- \lambda_{n} T} \phi_n (x)
\end{equation*}
Moreover, the operator $\Psi$ is self adjoint. Indeed,
\begin{align*}
    \left\langle \Psi \mathcal{G}_{1} , \mathcal{G}_{2}\right\rangle &= \left\langle \sum_{k = 1}^{\infty} \left\langle \mathcal{G}_{1},\phi_k \right\rangle  \mathrm{e}^{- \lambda_{k} T} \phi_k , \sum_{k = 1}^{\infty} \left\langle \mathcal{G}_{2},\phi_k \right \rangle \phi_k \right \rangle\\
    &= \sum_{k = 1}^{\infty} \left\langle \mathcal{G}_{1},\phi_k   \right\rangle \left\langle \mathcal{G}_{2},\phi_k \right\rangle \mathrm{e}^{- \lambda_{k} T}\\
    &=\left\langle \mathcal{G}_{1} ,\Psi  \mathcal{G}_{2}\right\rangle.
\end{align*}
\end{proof}
Next, we discuss the solvability of the BPP.
\begin{theorem}
Let $Y_{T}:=\left(y(T,\cdot),y_\Gamma(T,\cdot)\right)\in \mathbb{L}^2$ be a noise free measured data. Then the BPP in \eqref{2eq2.3} admits a solution if and only if
$$\sum_{k = 1}^{\infty} \mathrm{e}^{\lambda_{k} T} Y_{T,k}^{2}  < \infty , \qquad \text{ where }\, Y_{T,k}:= \left \langle Y_{T} , \phi_k \right \rangle,$$
in that case the solution is given by$$\mathcal{G} = \sum_{k = 1}^{\infty} \mathrm{e}^{\lambda_{k} T} Y_{T,n} \phi_k.$$
\end{theorem}
\begin{proof}
Let $\mathcal{G} \in \mathbb{L}^2$. Using \eqref{eq5.4}, $\Psi \mathcal{G}$ is given by
\begin{equation*}
    Y(T,x,\mathcal{G}) = \sum_{k = 1}^{\infty} \left\langle \mathcal{G}, \phi_k \right\rangle  \mathrm{e}^{- \lambda_{k} T} \phi_k (x)\qquad \text{ in } \Omega\times \Gamma.
\end{equation*}
Moreover, the output data can be expressed as follows
\begin{equation*}
    Y_{T} = \sum_{k = 1}^{\infty}Y_{T,k}  \phi_k.
\end{equation*}
Then, the Tikhonov functional \eqref{regpb} is represented by
\begin{align*}
   \mathcal{J}_\varepsilon(\mathcal{G})&= \frac{1}{2} \norm{\sum_{k = 1}^{\infty} \left( \left\langle \mathcal{G},\phi_k \right\rangle  \mathrm{e}^{- \lambda_{k} T} -  \left\langle Y_{T},\phi_k   \right\rangle\right)\phi_k}^{2} + \frac{\varepsilon}{2} \norm{\sum_{k = 1}^{\infty} \left\langle \mathcal{G},\phi_k \right\rangle \phi_k }^{2}\\
   &= \frac{1}{2} \sum_{k = 1}^{\infty} \left( \left\langle \mathcal{G},\phi_k \right> \mathrm{e}^{- \lambda_{k} T} - Y_{T,k} \right)^{2} + \varepsilon \left\langle \mathcal{G},\phi_k \right\rangle^2.
\end{align*}
Clearly, this functional achieves its minimum value if 
$$\left\langle \mathcal{G},\phi_k \right\rangle \mathrm{e}^{- \lambda_{k} T} - Y_{T,k} = 0.$$
Then, the $k$th Fourier coefficient of the unique
minimum of the Tikhonov functional \eqref{regpb} is given by 
$$ \left\langle \mathcal{G},\Phi_k \right\rangle =  \mathrm{e}^{ \lambda_{k} T} Y_{T,k}.$$
As a result, we obtain
$$\mathcal{G} = \sum_{k = 1}^{\infty}  \left\langle \mathcal{G},\phi_k \right\rangle \phi_k = \sum_{k = 1}^{\infty} \mathrm{e}^{\lambda_{k} T} Y_{T,k} \phi_k.$$
\end{proof}

\section{Numerical reconstruction of initial temperatures}\label{sec5}
Now we are in position to present the CG algorithm for reconstructing initial temperatures. We will consider both 1-D and 2-D initial temperatures reconstruction.

\smallskip
\begin{algorithm}[H]\label{alg1}
\SetAlgoLined
 Set $n=0$ and choose an initial temperature $\mathcal{G}_0$\;
 Solve the direct problem \eqref{eq1to4} to obtain $Y(t,x,\mathcal{G}_0)$\;
 Knowing the computed $Y(T,x,\mathcal{G}_0)$ and the measured $Y_T^\delta$, solve the adjoint problem \eqref{aeq1to5} to obtain $\Phi(t,x,\mathcal{G}_0)$\;
 Compute the gradient $p_0=\mathcal{J}_\varepsilon'\left(\mathcal{G}_0\right)$ using \eqref{2eq2.11}\;
 Solve the direct problem \eqref{eq1to4} with initial datum $p_n$ to obtain the solution $\Psi p_k$\;
 Compute the relaxation parameter $\displaystyle \alpha_n =\frac{\|\mathcal{J}_\varepsilon'(\mathcal{G}_n)\|_{\mathbb{L}^2}^2} {\|\Psi p_n\|_{\mathbb{L}^2}^2 + \varepsilon \|p_n\|_{\mathbb{L}^2}^2}$ \;
 Find the next iteration $\mathcal{G}_{n+1}=\mathcal{G}_n- \alpha_n p_n$\;
 Stop the iteration process if the stopping criterion $\mathcal{J}_\varepsilon(\mathcal{G}_{n+1}) <e_\mathcal{J}$ holds. Otherwise, set $n:=n+1$ and compute
 $$\gamma_n=\frac{\|\mathcal{J}_\varepsilon'(\mathcal{G}_n)\|_{\mathbb{L}^2}^2}{\|\mathcal{J}_\varepsilon'(\mathcal{G}_{n-1})\|_{\mathbb{L}^2}^2} \qquad \text{and} \qquad p_n=\mathcal{J}_\varepsilon'(\mathcal{G}_n)+\gamma_n p_{n-1},$$
 and go to Step 6;
 \caption{CG algorithm}
\end{algorithm}
\smallskip

In all numerical simulations, we assume that the measured output data contain a random noise. Instead of using the exact data $Y_T$, we generate the noisy data by the following formula
$$Y_T^\delta(x) =Y_T(x) + p \times\|Y_T\|_{\mathbb{L}^2} \times \mathrm{RandomReal[]},$$
where $p$ stands for the percentage of the noise level, and the function RandomReal[] produces random real numbers.

Comparing to static boundary conditions, the numerical implementation of dynamic boundary conditions is quite complicated. For this aim, the numerical solutions for the direct problems and the corresponding adjoint problems will be obtained by using the method of lines.

\subsection{1-D reconstruction of initial temperatures}
Let $\Omega = (0,\ell),\; \ell>0$ and $T>0$. We seek to reconstruct the unknown initial temperature $g(x)\in L^2(0,\ell)$ in the 1-D heat equation with dynamic boundary conditions
\begin{empheq}[left = \empheqlbrace]{alignat=2}
\begin{aligned}
&y_{t}(t, x)-y_{x x}(t, x)=0, &&\qquad(t,x)\in (0,T)\times (0,\ell)  , \\
&y_t(t, 0) - y_{x}(t, 0)=0, &&\qquad t\in (0,T), \\
&y_t(t, \ell) + y_{x}(t, \ell)=0, &&\qquad t\in (0,T), \\
&y(0,x)=g(x), \quad (y(0,0), y(0, \ell))=(a,b), &&\qquad x\in (0,\ell).
\label{1deq1to4}
\end{aligned}
\end{empheq}
The problem then is to find the initial temperature $(g(x),a,b)$ by using the final time overdetermination data.

As above, $Y(t,x,\mathcal{G}):=\left(y(t,x),y(t,0),y(t,\ell)\right)$ denote the solution of the system \eqref{1deq1to4}, where $\mathcal{G}:=(g, a, b)$. The Tikhonov functional is given by
\begin{align}\label{eq5.2}
\mathcal{J}_\varepsilon(\mathcal{G})&=\frac{1}{2} \left\|Y(T, \cdot,\mathcal{G})-Y_{T}^\delta \right\|_{\mathbb{L}^2}^2 +\frac{\varepsilon}{2} \|\mathcal{G}\|_{\mathbb{L}^2}^2 \\
& \hspace{-1cm}= \frac{1}{2} \left(\left\|y(T, \cdot)-y_{T}^\delta\right\|_{L^2(0,\ell)}^2 + \left|y(T,0)-y_T^{0,\delta}\right|^2 + \left|y(T,\ell)-y_T^{\ell,\delta}\right|^2 + \varepsilon \|\mathcal{G}\|_{\mathbb{L}^2}^2\right), \notag
\end{align}
where $Y_{T}^\delta:=\left(y_{T}^\delta, y_T^{0,\delta}, y_T^{\ell,\delta}\right) \in \mathbb{L}^2$ is the observation data.

In this 1-D case, the governing operator takes the form
$$
\mathcal{A}_{0}:=\left(\begin{array}{ccc}
\partial_{xx} & 0 & 0 \\
\left.\partial_x\right|_{x=0} & 0 & 0 \\
-\left.\partial_x\right|_{x=\ell} & 0 & 0
\end{array}\right), \qquad D\left(\mathcal{A}_{0}\right)=\mathbb{H}^{2}=\left\{(f, f(0), f(\ell)): f \in H^{2}(0, \ell)\right\}.
$$

\begin{remark}
It is well-known that the decay to zero of singular values $\sigma_k$ of the input-output operator $\Psi$ indicates the degree of ill-posedness of the corresponding inverse problem. It has been shown in \cite[Proposition 5.3.1]{Kh'20} that the eigenvalues of $-\mathcal{A}_0$ are of order $\lambda_k=\mathcal{O}\left(k^2\right)$. It follows that the singular values $\sigma_k=\mathrm{e}^{-\lambda_{k} T}$ of $\Psi$ are of order $\mathcal{O}\left(\mathrm{e}^{-k}\right)$. This means that the BPP we consider is severely exponentially ill-posed.
\end{remark}

The adjoint system corresponding to \eqref{1deq1to4} is given by
\begin{empheq}[left = \empheqlbrace]{alignat=2}
\begin{aligned}
& \varphi_{t}(t, x)+\varphi_{x x}(t, x)=0, &&\hspace{-1cm} (t,x)\in (0,T)\times (0,\ell)  , \\
& \varphi_t(t, 0) + \varphi_{x}(t, 0)=0, && t\in (0,T), \\
& \varphi_t(t, \ell) - \varphi_{x}(t, \ell)=0, && t\in (0,T), \\
& \varphi(T, x)=y(T,x)-y_T^\delta(x), && x\in (0,\ell),\\
& (\varphi(T,0), \varphi(T, \ell))=\left(y(T,0)-y_T^{0,\delta},y(T,\ell)-y_T^{\ell,\delta}\right).
\label{1daeq1to4}
\end{aligned}
\end{empheq}
Similarly to Proposition \ref{2prop2.6}, we obtain the following formula for the gradient of $\mathcal{J}_\varepsilon$ 
\begin{equation*}
\mathcal{J}_\varepsilon'(\mathcal{G})(x)= \left(\varphi(0,x,\mathcal{G}),\varphi(0,0,\mathcal{G}),\varphi(0,\ell,\mathcal{G})\right) + \varepsilon \mathcal{G}.
\end{equation*}

Next, we discretize the space domain $[0,\ell]$ into a uniform grid $(x_j)_{j=0}^{N_x}$ of step $\Delta x=\frac{\ell}{N_x}$ using finite differences. Denoting by $y_j(t)=y(t,x_j)$, the the second-order derivative of $y$ in system \eqref{1deq1to4} is approximated as follows
$$y_{xx}(t,x_j) \approx \frac{y_{j-1}(t)-2 y_j(t) + y_{j+1}(t)}{(\Delta x)^2}, \qquad j=\overline{1, N_x-1}.$$
The first-order derivatives on the boundary can be approximated by
\begin{align*}
    y_x(t,0) &\approx \frac{y_1(t)-y_0(t)}{\Delta x}\\
    y_x(t,\ell) &\approx \frac{y_{N_x}(t)-y_{N_x-1}(t)}{\Delta x}.
\end{align*}
Now, it suffices to numerically solve the resulting ordinary differential system. This is done with help of the \texttt{Wolfram} language.

In the following, we will test the stability of our algorithm using some numerical experiments. For convenience, in all experiments, the system parameters are taken as follows
$$T=0.03, \quad \ell=1, \quad N_x=25, \quad a=b=0.$$
For each example, an illustration is provided for the convergence error
$$e(n,g_n):=\|\Psi (g_n,0,0) -Y_T\|_{L^2(0,1)\times \mathbb{R}^2}^2,$$
and the accuracy error $$E(n,g_n):=\|g-g_n\|_{L^2(0,1)}$$
in terms of the iteration number $n$.

\subsection*{Example 1}
We take the exact initial temperature as
$$g(x)= \frac{7}{24} \left(\sin(\pi x) +1-x\right) \left(\sin(\pi x) + \sqrt{x} \right), \quad x\in (0,1).$$

Let us set the starting iteration as $g_0=0$, the
regularization parameter $\varepsilon=10^{-8}$, and stopping parameter $e_\mathcal{J}=10^{-6}$. The algorithm stops at iteration $n=9$ for $p\in \{1\%,3\%,5\%\}$.

\begin{figure}[H]
    \centering
    {{\includegraphics[width=6.1cm]{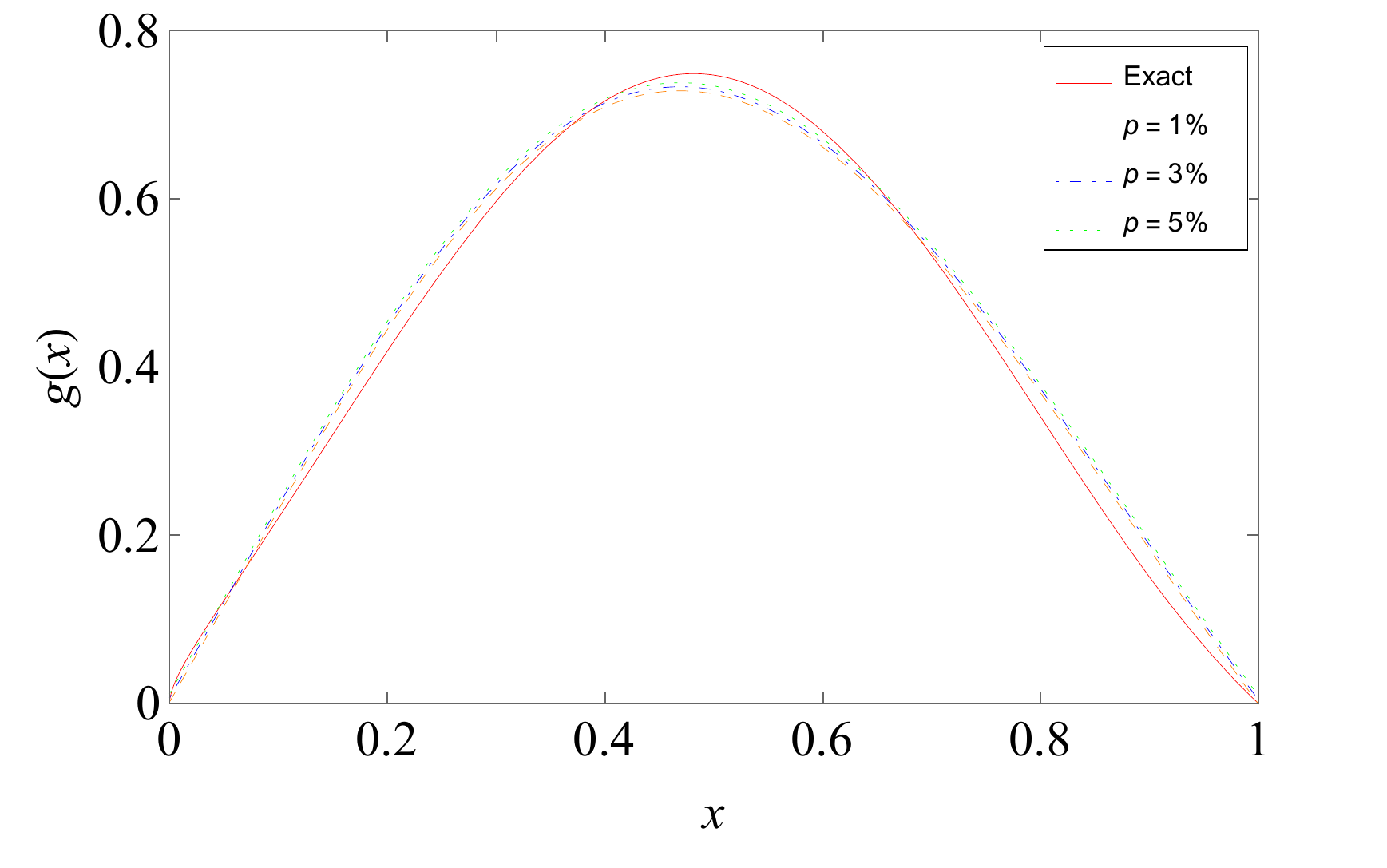} }}%
    \hspace{0.1cm}
    {{\includegraphics[width=6.1cm]{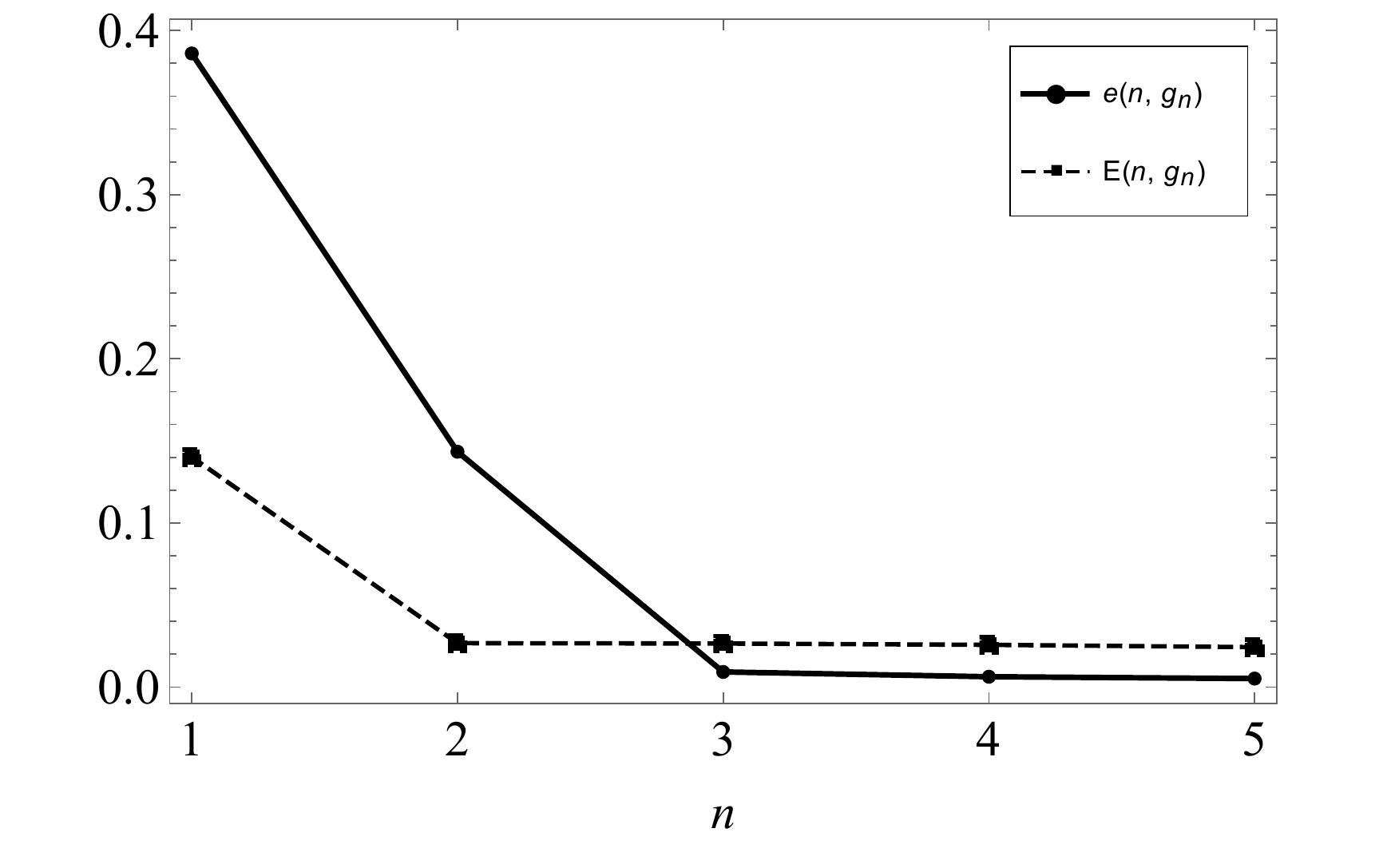} }}%
    \caption{Exact and recovered $g(x)$ from noisy data (left) and corresponding errors for first iterations for noise free data (right)}%
    \label{fig01}%
\end{figure}

\subsection*{Example 2}
We choose the exact initial temperature as
$$g(x)=6 (1-x)\log \left(1+x^2\right), \; x\in (0,1).$$

The initial iteration is $f_0=0$. The regularization parameter is $\varepsilon=10^{-8}$ and the stopping parameter is $e_\mathcal{J}=10^{-6}$. The algorithm stops at iterations $n\in \{39,42,43\}$ for $p\in \{1\%,3\%,5\%\}$, respectively.

\begin{figure}[H]
    \centering
    {{\includegraphics[width=6.1cm]{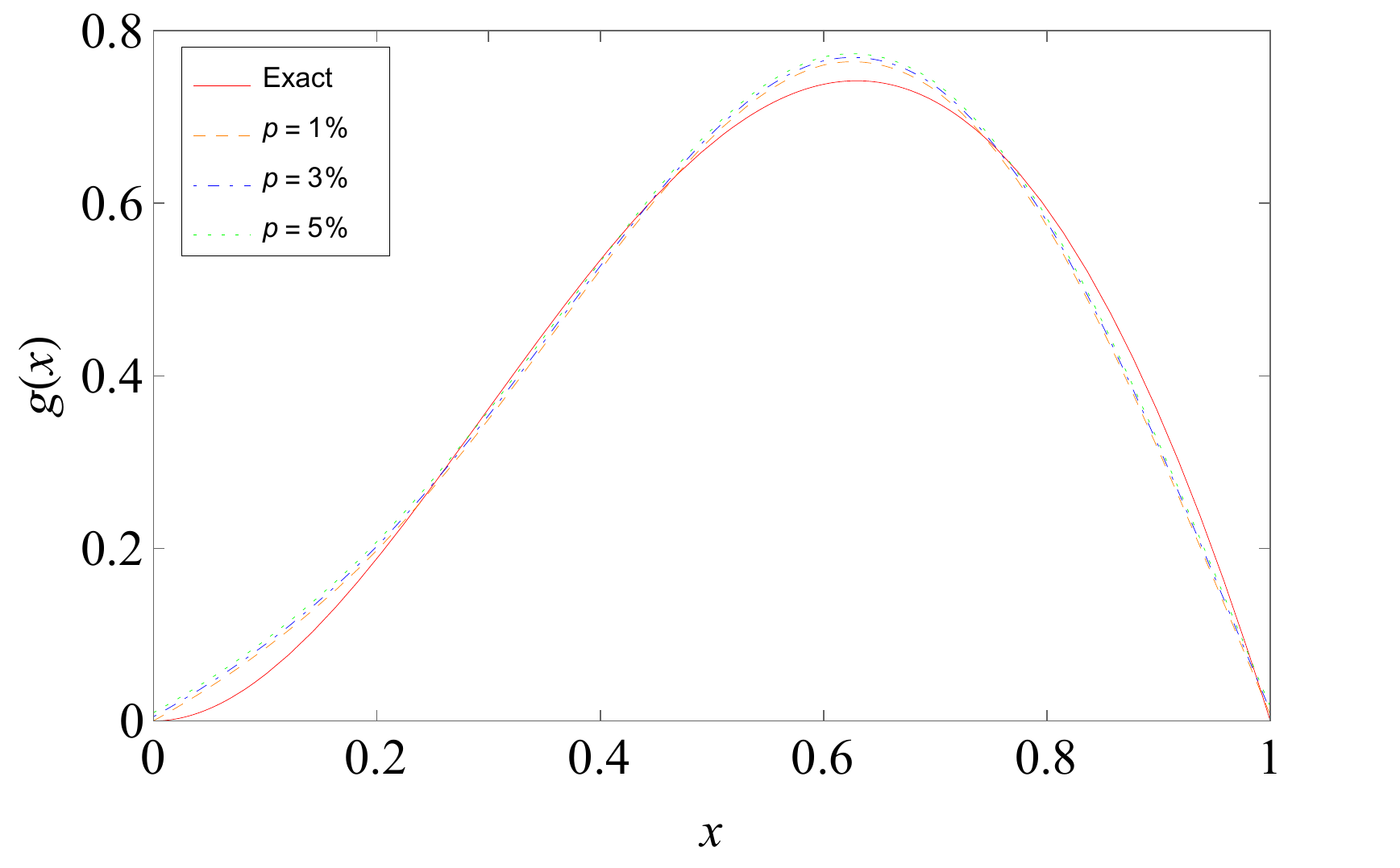} }}%
    \hspace{0.1cm}
    {{\includegraphics[width=6.1cm]{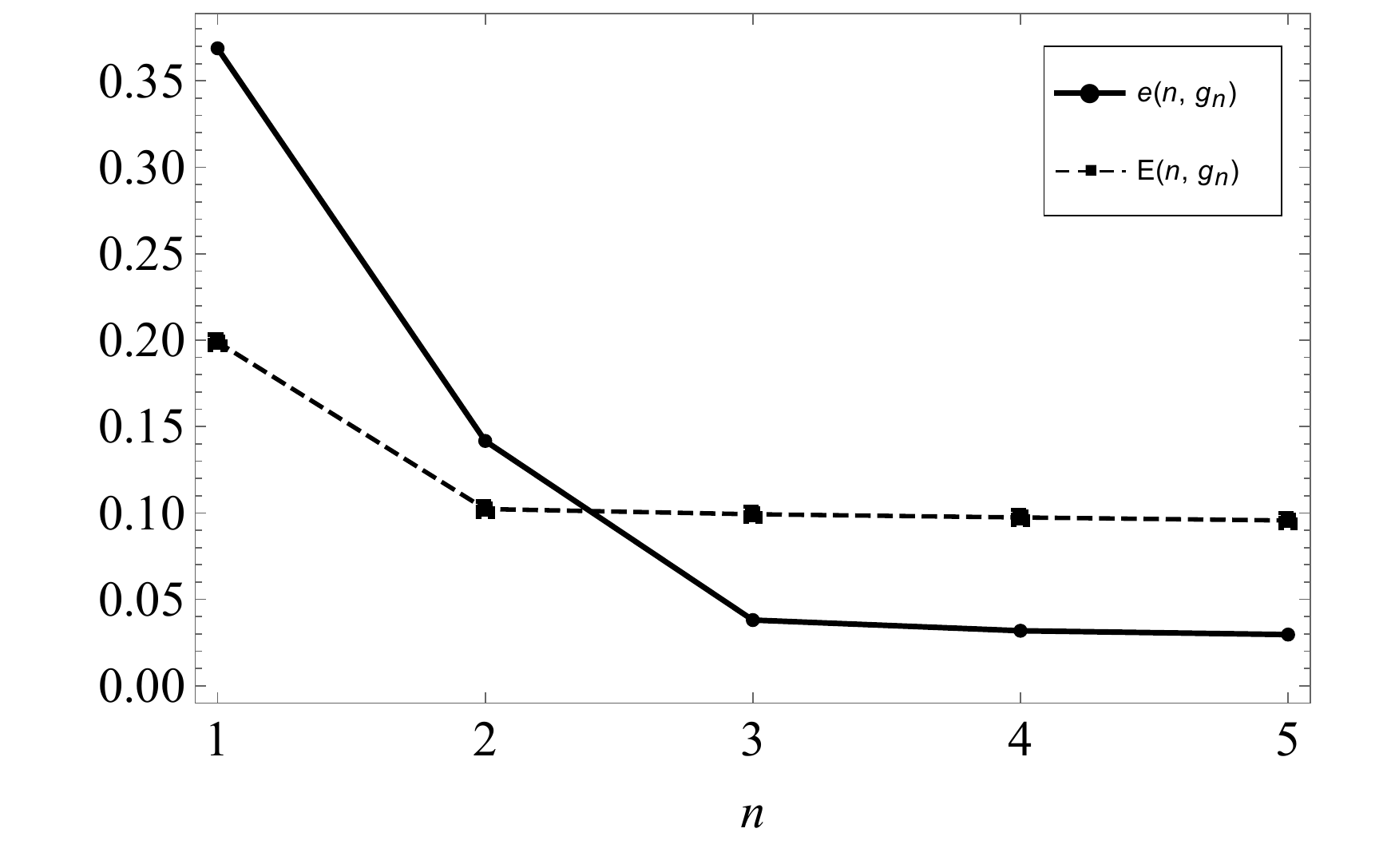} }}%
    \caption{Exact and recovered $g(x)$ from noisy data (left) and corresponding errors for first iterations for noise free data (right)}%
    \label{fig02}%
\end{figure}

\begin{remark}
The above numerical tests demonstrate the stability, the accuracy and the speed of the designed CG algorithm for the reconstruction of 1-D unknown initial temperatures in heat equation with dynamic boundary conditions. It is clear that the restoration of the initial temperature $g(x)$ is not much affected by random noise $p$. Furthermore, Figs. \ref{fig01} and \ref{fig02} show the decay of the convergence error and the accuracy error as the iteration number increases. This fact validates the convergence of the regularized solutions to the exact initial temperatures.
\end{remark}

\subsection{2-D reconstruction of initial temperatures}

Let $\Omega=\left\{(x_1,x_2)\in \mathbb{R}^2 \colon x_1^2 + x_2^2 <1 \right\}$ be the unit disk. Its boundary $\Gamma=\left\{(x_1,x_2)\in \mathbb{R}^2 \colon x_1^2 + x_2^2 =1 \right\}=\mathbb{S}^1$ is the unit circle. In polar coordinates,
$$\Omega=\{(r\cos\theta, r\sin\theta) \colon r\in [0,1),\theta \in [0,2\pi)\}\text{ and }\Gamma=\{(\cos\theta, \sin\theta) \colon \theta \in [0,2\pi)\}.$$
The gradient of any smooth function $y$ on $\Omega$ is given locally by
$$\nabla y=\partial_r y\, e_r +\frac{1}{r} \partial_\theta y\, e_\theta,$$
where $e_r$ and $e_\theta$ are the unit vectors in $r$ and in $\theta$ directions respectively. At any point on $\Gamma$, the outer unit normal vector is $\nu=e_r$. Then, by the previous formula the normal derivative of $y$ is given by
$$\partial_\nu y=\nabla y \cdot \nu \rvert_{\Gamma} = \partial_r y\, e_r \cdot e_r \rvert_{\Gamma}=\partial_r y \rvert_{r=1}.$$
The standard Laplacian $\Delta$ on $\Omega$ takes the form
$$\Delta y=\partial^2_r y+\frac{1}{r} \partial_r y+\frac{1}{r^2} \partial^2_\theta y.$$
For any $(x_1,x_2)=(\cos\theta, \sin\theta)\in \Gamma$, the tangent space at $(x_1,x_2)$ is $T_{(x_1,x_2)} \Gamma=\text{span}{(-\sin\theta, \cos\theta)}$. Denoting $\partial_\theta:=(-\sin\theta, \cos\theta)$, we have $|\partial_\theta|=1$. Then the metric tensor here is a scalar $\mathrm{g}=1$. Using the local formula \eqref{eqlb}, the Laplace-Beltrami operator on $\Gamma$ is given by
$$\Delta_\Gamma y_\Gamma =\partial^2_\theta y_\Gamma .$$
Consequently, a simplified version of system \eqref{eq1to4} in polar coordinates reads as follows 
\begin{empheq}[left = \empheqlbrace]{alignat=2}
\begin{aligned}
&\partial_t y -\left[\partial^2_r y+\frac{1}{r} \partial_r y+\frac{1}{r^2} \partial^2_\theta y\right]=0, &&\qquad 0<r<1, \; 0\le \theta \le 2 \pi \\
&\partial_t y_{\Gamma} -\partial^2_\theta y_\Gamma + \partial_r y \rvert_{r=1}=0, &&\qquad 0\le \theta \le 2 \pi, \\
& y_{\Gamma}(t,\theta) = y(t,r,\theta)\rvert_{r=1}, &&\qquad 0\le \theta \le 2 \pi,\\
&(y,y_{\Gamma})\rvert_{t=0}=(g, g_\Gamma), &&\qquad 0<r<1, \; 0\le \theta \le 2 \pi.
\label{2deq1to4}
\end{aligned}
\end{empheq}
Our purpose is to determine the initial temperature $g(r,\theta)$ using the final time overdetermination data. The corresponding adjoint system is given by
\begin{empheq}[left = \empheqlbrace]{alignat=2}
\begin{aligned}
& -\partial_t \varphi -\left[\partial^2_r \varphi+\frac{1}{r} \partial_r \varphi+\frac{1}{r^2} \partial^2_\theta \varphi\right]=0, &&\qquad 0<r<1, \; 0\le \theta \le 2 \pi \\
& -\partial_t \varphi_{\Gamma} -\partial^2_\theta \varphi_\Gamma + \partial_r \varphi \rvert_{r=1}=0, &&\qquad 0\le \theta \le 2 \pi, \\
& \varphi_{\Gamma}(t,\theta) = \varphi(t,r,\theta)\rvert_{r=1}, &&\qquad 0\le \theta \le 2 \pi,\\
& \varphi \rvert_{t=T}= y(T, \cdot, \mathcal{G})-y_{T}^\delta, &&\qquad 0<r<1, \; 0\le \theta \le 2 \pi, \\
& \varphi_{\Gamma}=y_\Gamma(T, \cdot, \mathcal{G})-y_{T,\Gamma}^\delta) , &&\qquad 0<r<1, \; 0\le \theta \le 2 \pi. \nonumber
\end{aligned}
\end{empheq}

Then the gradient formula for $\mathcal{J}_\varepsilon$ at $\mathcal{G}=(g,g_\Gamma)$ is given by
\begin{equation*}
\mathcal{J}_\varepsilon'(\mathcal{G})(r,\theta)= \left(\varphi(0,r,\theta,\mathcal{G}),\varphi(0,r,\theta,\mathcal{G})\rvert_{r=1}\right) + \varepsilon \mathcal{G}, \qquad 0<r<1, \; 0\le \theta \le 2 \pi.
\end{equation*}

To solve direct and adjoint problems, we discretize the space domain $\Omega$ into a uniform grid $\left(r_j, \theta_n \right), 0 \le j \le N_r, \, 0 \le n \le N_\theta$, of steps $\Delta r=\frac{1}{N_r}$ and $\Delta \theta=\frac{2\pi}{N_\theta}$ using finite differences (Fig. \ref{figpg}).

\begin{figure}[H] 
\centering
\includegraphics[scale=0.2]{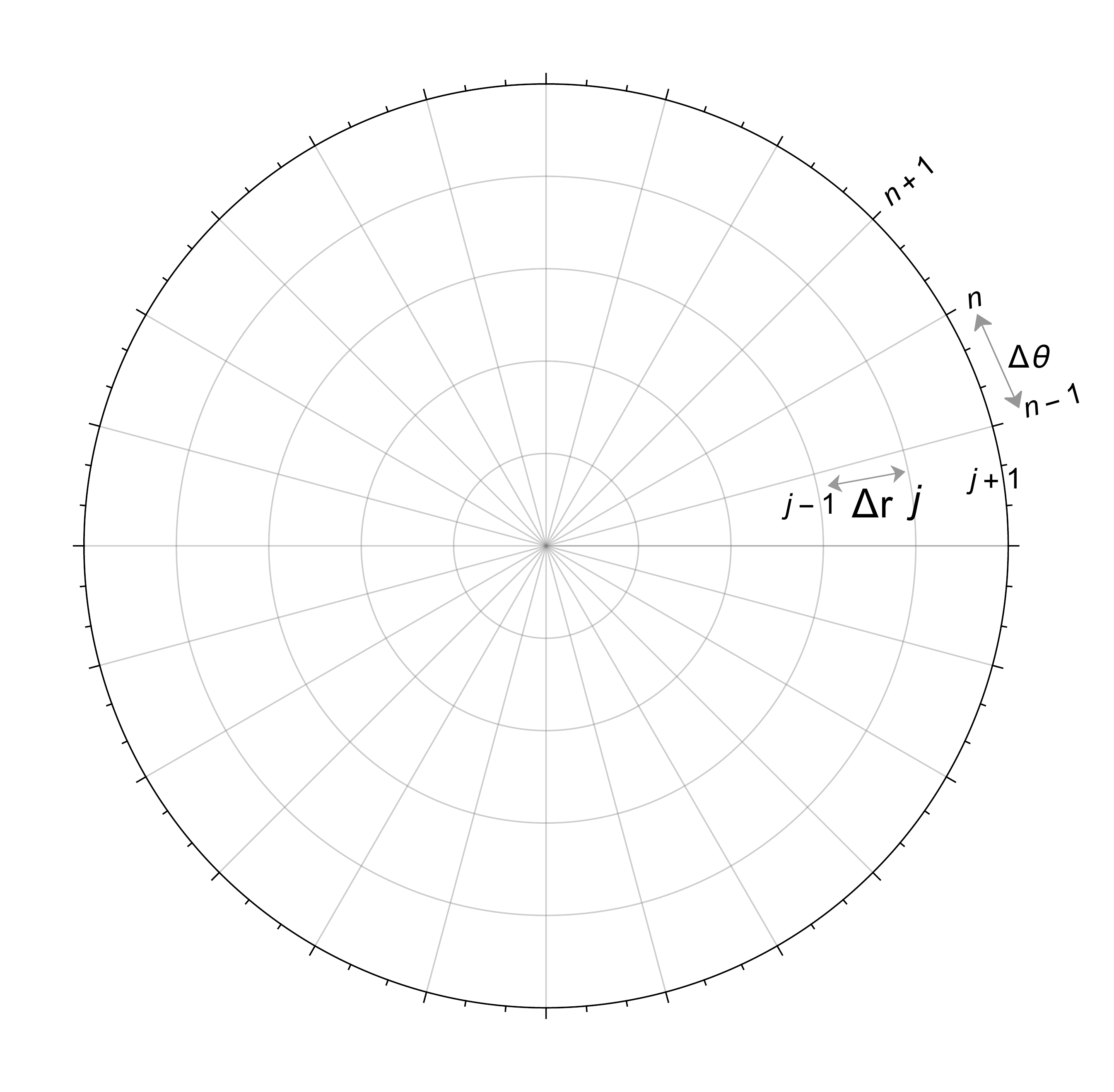}
\caption{Polar grid for the disk.} \label{figpg}
\end{figure}

Multiplying both sides of $\eqref{2deq1to4}_1$ by $r^2$, we obtain
\begin{equation}\label{eqdg}
    r^2\partial_t y -\left[r^2\partial^2_r y + r \partial_r y + \partial^2_\theta y\right]=0.
\end{equation}
Denoting by $y_{j,n}(t)=y(t,r_j, \theta_j)$, the the second-order partial derivatives of $y$ in system \eqref{2deq1to4} are approximated as follows 
\begin{align*}
    \partial_r^2 y(t,r_j, \theta_n) & \approx \frac{y_{j-1,n}(t)-2 y_{j,n}(t) + y_{j+1,n}(t)}{(\Delta r)^2}, \qquad j=\overline{1, N_r-1}\\
    \partial_\theta^2 y(t,r_j, \theta_n) & \approx \frac{y_{j,n-1}(t)-2 y_{j,n}(t) + y_{j,n+1}(t)}{(\Delta \theta)^2}, \qquad n=\overline{1, N_\theta-1}.
\end{align*}
The partial derivatives on the boundary are approximated by
\begin{align*}
    \partial_r y \rvert_{r=1} &\approx \frac{y_{N_r,n}(t)-y_{N_r-1,n}(t)}{\Delta r}, \qquad n=\overline{0, N_\theta}\\
    \partial_\theta^2 y_\Gamma(t,\theta_n) & \approx \frac{y_{N_r,n-1}(t)-2 y_{N_r,n}(t) + y_{N_r,n+1}(t)}{(\Delta \theta)^2}, \qquad n=\overline{1, N_\theta-1}.
\end{align*}
Note that the boundary condition for $\theta$ is given by periodicity $y(t,r,0)=y(t,r,2\pi)$ which corresponds to $y_{j,0}=y_{j,N_\theta}$. Since the system \eqref{2deq1to4} is singular at $r=0$, some boundedness condition should be added to avoid the blow up. We refer to \cite{Ya'13} for more details.

Henceforth, we take the final time $T=0.01$ and the mesh parameters $N_r=N_\theta=25$. We define the error surface as the difference between the exact initial temperature and the recovered one.
\subsection*{Example 1}
We take the exact initial temperature as the separate variables function
$$g(r,\theta)= \sin(\pi r) \sin \left(\frac{\theta}{2}\right), \quad r\in (0,1), \; \theta\in [0,2\pi].$$

We set the initial iteration as $g_0=0$, the
regularization parameter $\varepsilon=10^{-8}$, and the stopping parameter $e_\mathcal{J}=5.82\times 10^{-5}$. The algorithm stops at iteration $n=32$ for a noise level $p=1\%$.

\begin{figure}[H]
    \centering
    {{\includegraphics[width=6.1cm]{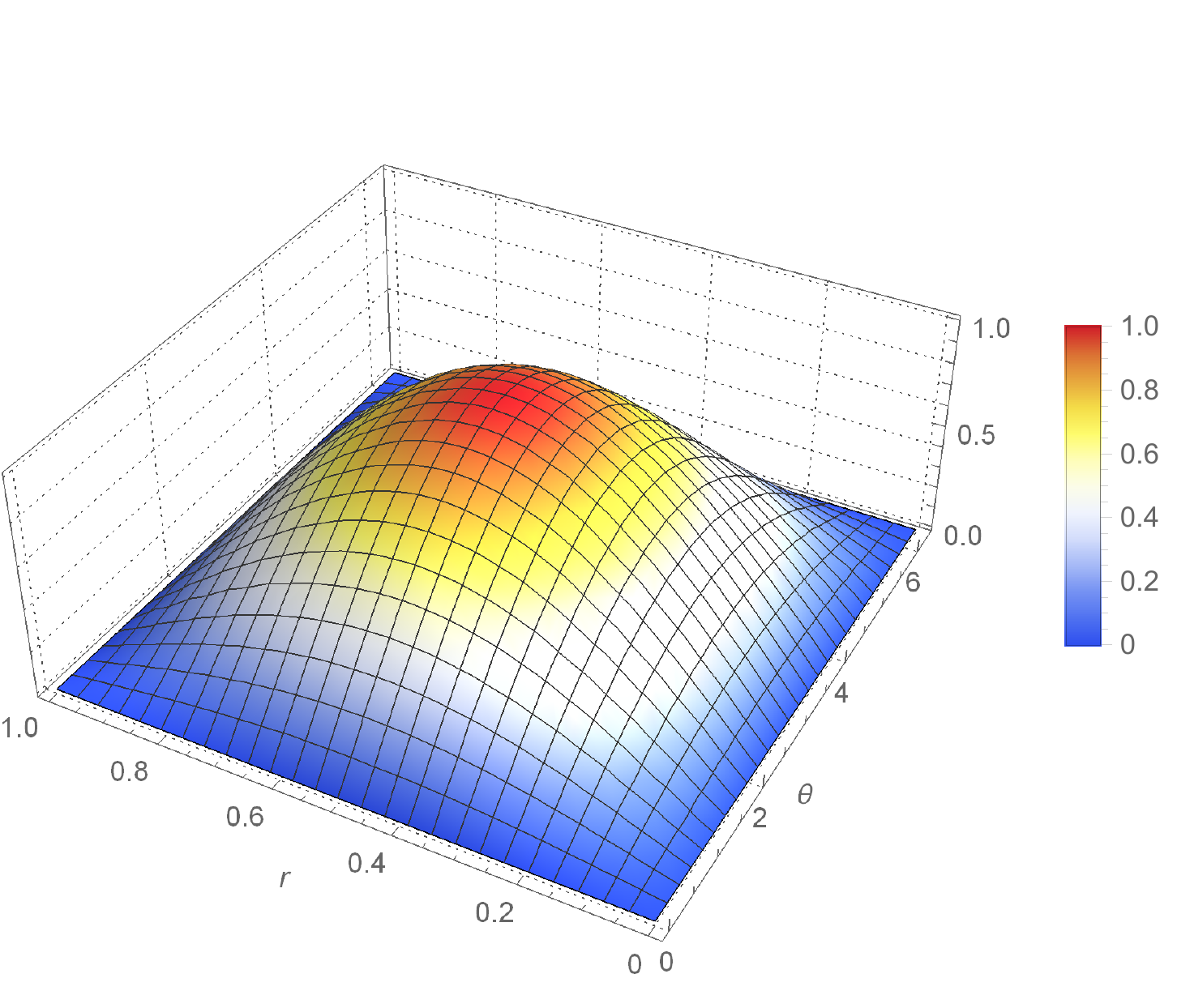} }}%
    \hspace{0.2cm}
    {{\includegraphics[width=6.1cm]{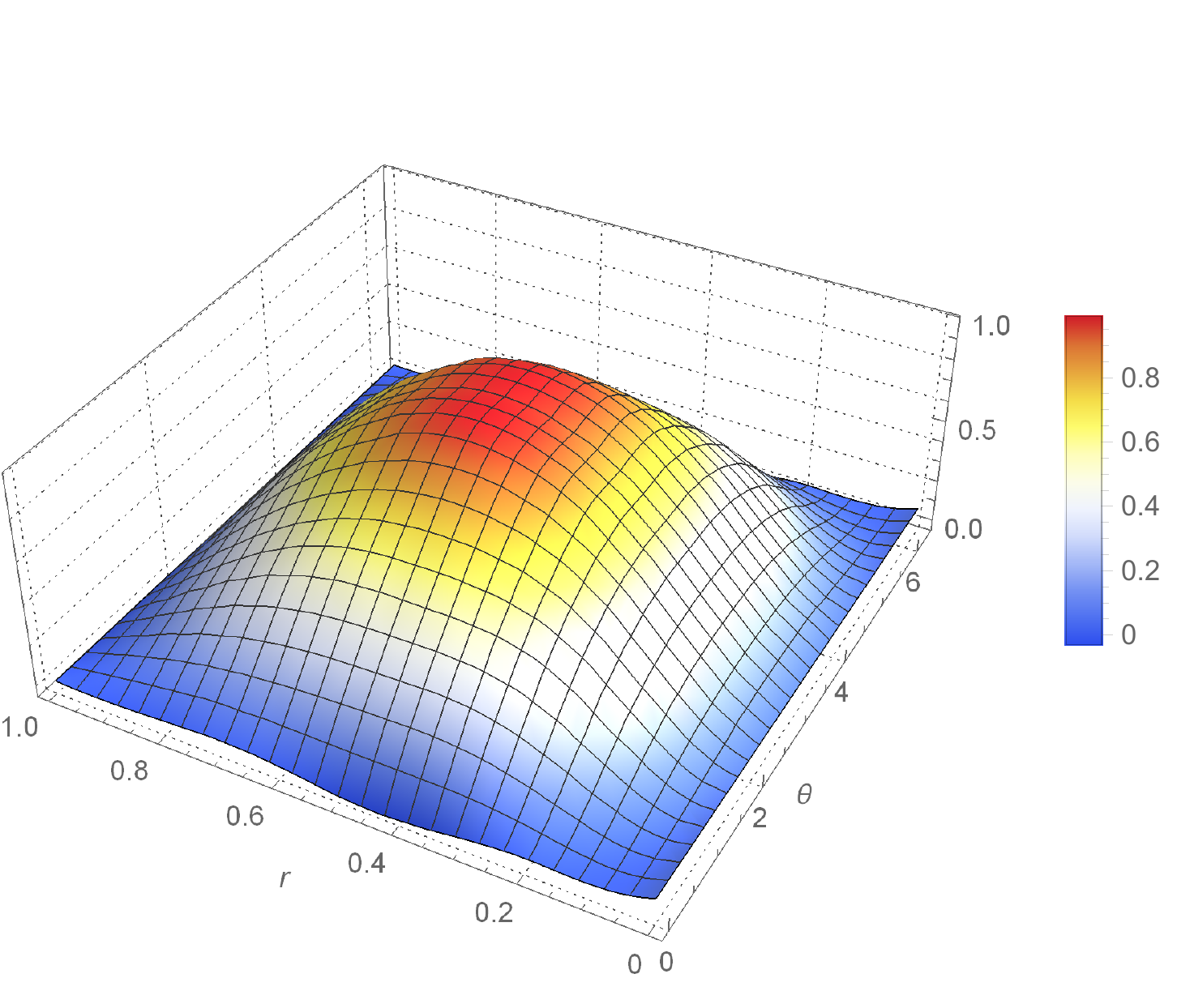} }}%
    \caption{Exact $g(r,\theta)$ (left) and recovered one from noisy data (right) for Example 1}%
    \label{fig11}%
\end{figure}

\begin{figure}[H] 
\centering
\includegraphics[scale=0.5]{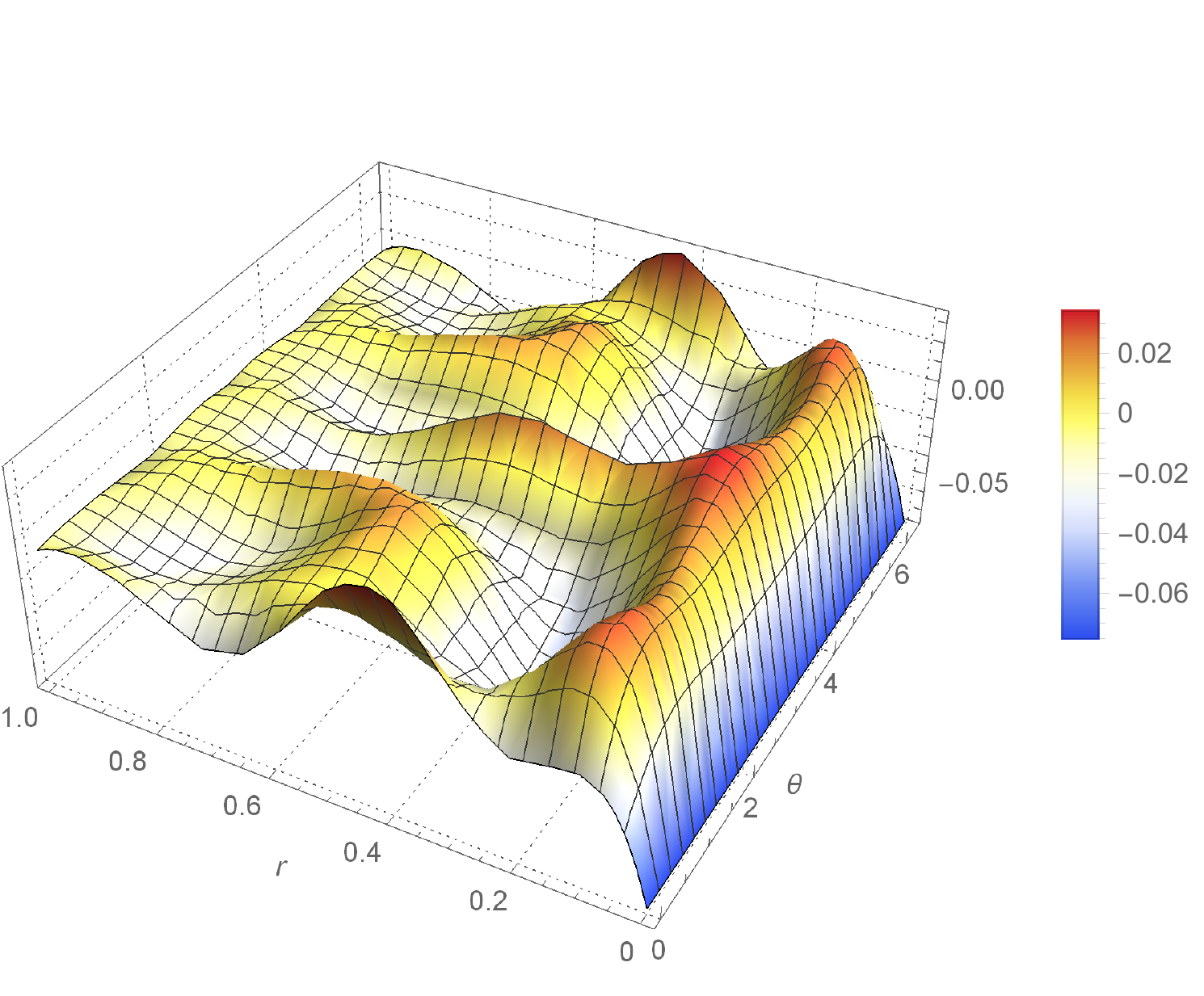}
\caption{Error surface for the noise level $p=1\%$.} \label{figerr11}
\end{figure}

\subsection*{Example 2}
We take the exact initial temperature
$$g(r,\theta)= [r(1-r)]^2 \cos \left(\frac{r\theta}{4}\right), \quad r\in (0,1), \; \theta\in [0,2\pi].$$

We set the initial iteration as $g_0=0$, the
regularization parameter $\varepsilon=10^{-8}$, and the stopping parameter $e_\mathcal{J}=2.92\times 10^{-7}$. The algorithm stops at iteration $n=60$ for a noise level $p=1\%$.

\begin{figure}[H]
    \centering
    {{\includegraphics[width=6.1cm]{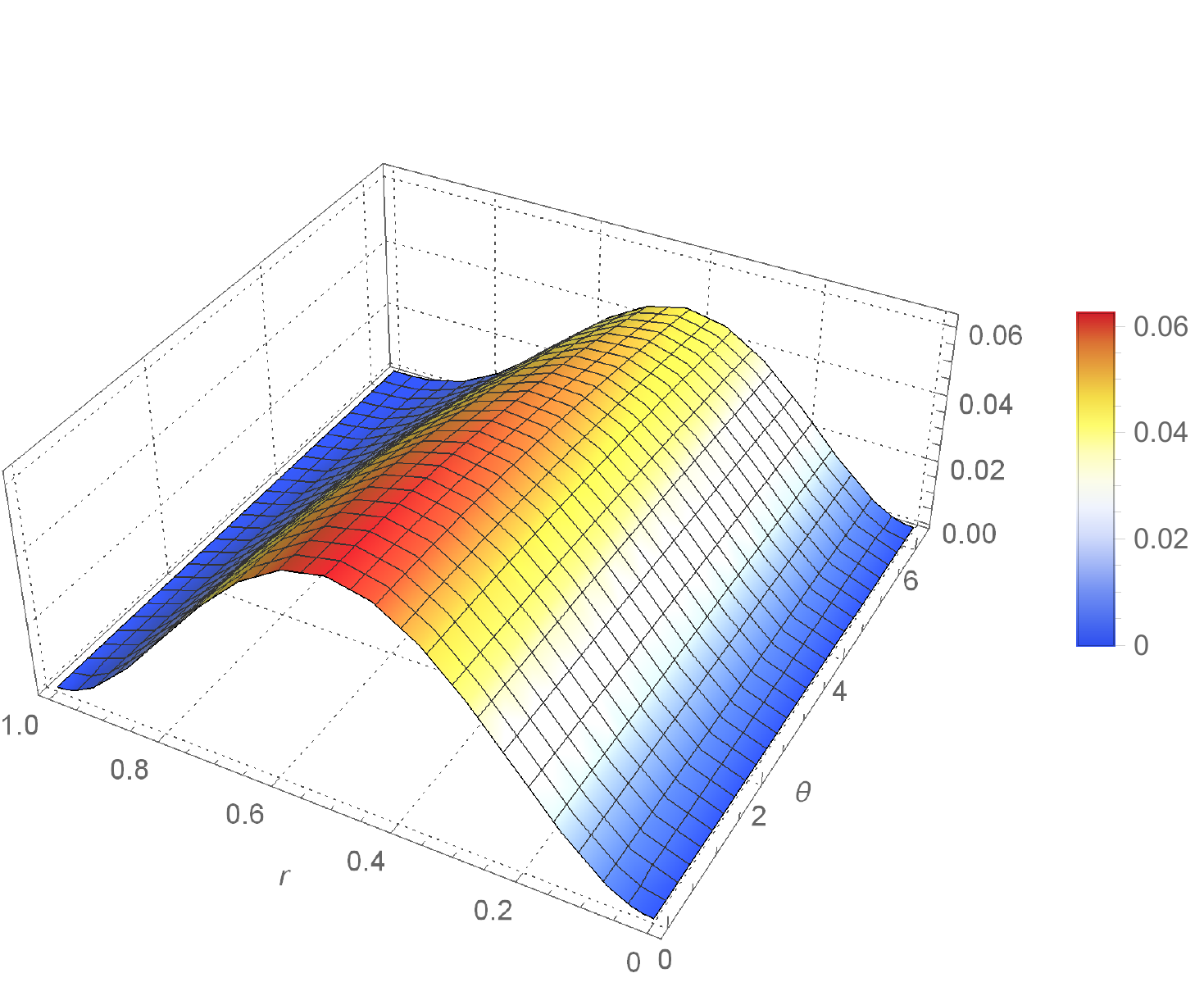} }}%
    \hspace{0.2cm}
    {{\includegraphics[width=6.1cm]{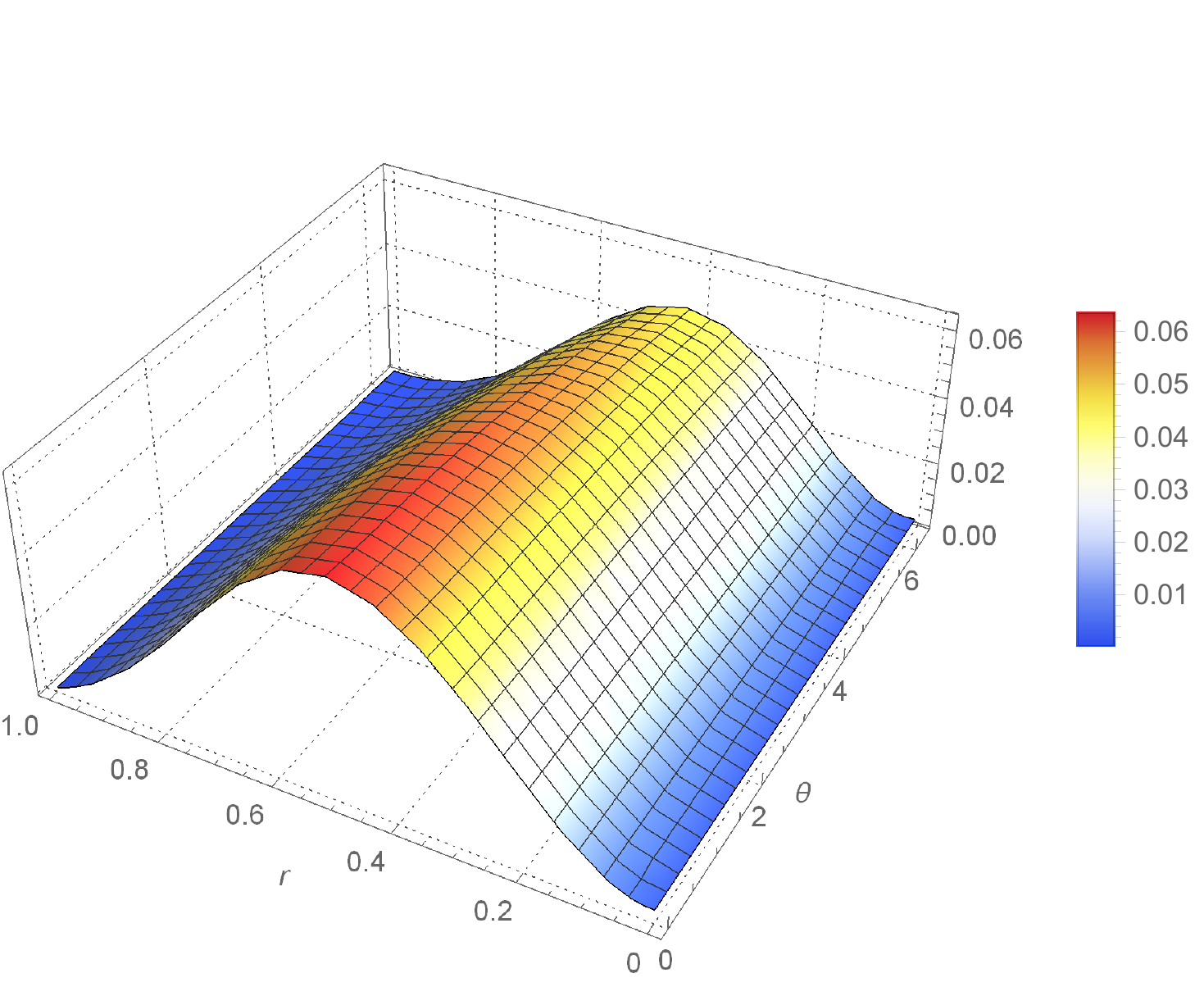} }}%
    \caption{Exact $g(r,\theta)$ (left) and recovered one from noisy data (right) for Example 2.}%
    \label{fig22}%
\end{figure}

\begin{figure}[H] 
\centering
\includegraphics[scale=0.5]{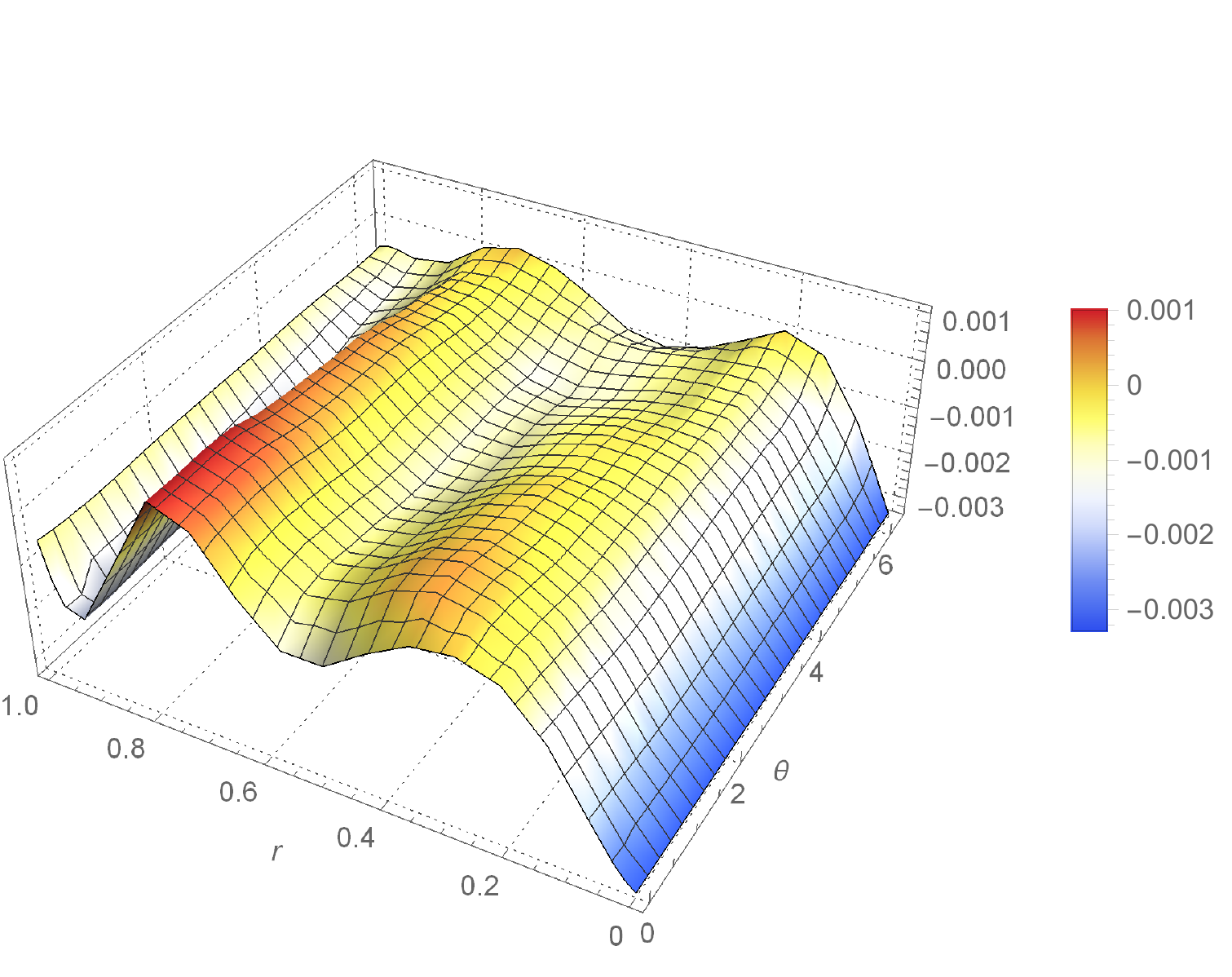}
\caption{Error surface for the noise level $p=1\%$.} \label{figerr22}
\end{figure}

\begin{remark}
In the previous experiments, the optimal regularization parameter $\varepsilon^*$ can be chosen in terms of the noise $\delta$ by the rule $\varepsilon<1$ and $\frac{\delta^2}{\varepsilon}<1$, while the optimal stopping parameter $e^*_\mathcal{J}$ can be estimated by the iteration number $n$ from the convergence the accuracy errors. See \cite{HR'21} for more details.
\end{remark}

\section{Conclusion}\label{sec6}
The inverse problem of identifying unknown initial temperatures in heat conduction models from the final time data is of interest in many engineering problems and industrial applications. Such models with static boundary conditions have been extensively studied and the theory is well developed, while papers dealing with heat models with dynamic boundary conditions are few and still at an earlier stage.

In this paper, we have investigated the inverse problem of recovering unknown initial temperatures of a heat equation with dynamic boundary conditions. Using an optimal control approach, the existence, uniqueness and stability of the minimizers are discussed for a general case of bounded domains $\Omega \subset \mathbb{R}^N$ ($N\geq 1$), and the numerical results are obtained by the conjugate gradient method for the 1-D and 2-D cases. The presented method has the advantage of yielding a monotone scheme and then fast numerical results. We expect that the proposed numerical approach is still applicable for such a problem in the 3-D case of system \eqref{eq1to4}, i.e. $N=3$. However, the geometric complexity of the domain $\Omega$ and the presence of the Laplace-Beltrami operator on the boundary will make the analysis more difficult. The study of such a problem would be of much interest.

\end{document}